\documentclass{amsart}
\usepackage[left=2.45cm, top=2.45cm,bottom=2.45cm,right=2.45cm]{geometry}
\usepackage{amsfonts}
\usepackage{latexsym}
\usepackage{amssymb}
\usepackage{amsmath}
\usepackage{color}
\usepackage{bbm}
\usepackage{tikz}
\usepackage{enumerate}
\usepackage{mathrsfs}

\newcommand{\R}{\mathbb R}
\newcommand{\N}{\mathbb N}

\newcommand{\E}{\mathbb E}
\newcommand{\rate}{\mathbb I}
\newcommand{\Pro}{\mathbb P}

\newcommand{\Var}{\mathrm{Var}}
\newcommand{\Cov}{\mathrm{Cov}}

\def\dint{\textup{d}}
\newcommand{\SSS}{\ensuremath{{\mathbb S}}}

\newcommand{\B}{\ensuremath{{\mathbb B}}}

\newcommand{\todistr}{\overset{d}{\underset{n\to\infty}\longrightarrow}}
\newcommand{\toprob}{\overset{\Pro}{\underset{n\to\infty}\longrightarrow}}

\DeclareMathOperator{\vol}{vol}

\newtheorem{thm}{Theorem} 

\newtheorem{lemma}[thm]{Lemma}

\newtheorem{rmk}[thm]{Remark}
\newtheorem{thmalpha}{Theorem}

\def\bC{\mathbf{C}}

\def\bP{\mathbf{P}}

\def\bU{\mathbf{U}}

\def\bW{\mathbf{W}}
\def\bX{\mathbf{X}}
\def\bY{\mathbf{Y}}

\def\XX{\mathbb{X}}
\def\YY{\mathbb{Y}}

\usepackage{nomencl}
\makenomenclature

\begin{document}


\title[Limit theorems for random vectors in $\ell_p^n$-balls]{High-dimensional limit theorems\\ for random vectors in $\ell_p^n$-balls. II}

\author[Z. Kabluchko]{Zakhar Kabluchko}
\address{Institut f\"ur Mathematische Stochastik, Westf\"alische Wilhelms-Universit\"at M\"unster, Germany} \email{zakhar.kabluchko@uni-muenster.de}

\author[J. Prochno]{Joscha Prochno}
\address{Institut f\"ur Mathematik \& Wissenschaftliches Rechnen, Karl-Franzens-Universit\"at Graz, Austria} \email{joscha.prochno@uni-graz.at}

\author[C. Th\"ale]{Christoph Th\"ale}
\address{Faculty of Mathematics, Ruhr University Bochum, Germany} \email{christoph.thaele@rub.de}


\begin{abstract}
In this article we prove three fundamental types of limit theorems for the $q$-norm of random vectors chosen at random in an $\ell_p^n$-ball in high dimensions. We obtain a central limit theorem, a moderate deviations as well as a large deviations principle when the underlying distribution of the random vectors belongs to a general class introduced by Barthe, Gu\'edon, Mendelson, and Naor. It includes the normalized volume and the cone probability measure as well as projections of these measures as special cases. Two new applications to random and non-random projections of $\ell_p^n$-balls to lower-dimensional subspaces are discussed as well. The text is a continuation of [Kabluchko, Prochno, Th\"ale: High-dimensional limit theorems for random vectors in $\ell_p^n$-balls, Commun. Contemp. Math. (2019)].
\end{abstract}

\maketitle

\section{Introduction and main results}

The study of high-dimensional geometric structures and particularly of convex bodies
has received considerable attention in the last decade. In parts, this was triggered by modern applications
in high-dimensional statistics, machine learning, and numerical analysis. Many of the deep discoveries are of a probabilistic
flavor or have been obtained by means of novel and powerful probabilistic methods.
It therefore comes as no surprise that (central) limit theorems
have been obtained for various quantities that appear in high-dimensional stochastic geometry or the asymptotic theory of convex bodies. Probably the first high-dimensional
central limit theorem is known as the Poincar\'e-Maxwell-Borel Lemma (see, e.g., \cite{DiaconisFreedman,Stam82}).
It shows that the distribution of the first $k$ coordinates of a point chosen
uniformly at random from the $n$-dimensional Euclidean ball or sphere converges
to a $k$-dimensional Gaussian distribution, as the dimension $n$ of the ambient space tends to infinity. The
most prominent result of the past $15$ years is arguably Klartag's central limit theorem for isotropic convex bodies \cite{KlartagCLT}, showing that most $k$-dimensional marginals of random points chosen uniformly at random from a convex body are approximately Gaussian. Many more deep central limit phenomena have been discovered in the recent past. Among others, there is a central limit theorem for the volume of
convex hulls of Gaussian random vectors obtained by B\'ar\'any and Vu in \cite{BV2007} or Reitzner's central limit theorems for the volume and the number of $i$-dimensional faces of random polytopes in smooth convex bodies \cite{R2005} that were obtained when the number of random points tends to infinity (see also  B\'ar\'any and Th\"ale \cite{BaranyThaele} and Th\"ale, Turchi, and Wespi \cite{ThaeleTurchiWespi} for results  about general intrinsic volumes).
There is a central limit theorem due to Paouris, Pivovarov, and Zinn \cite{PPZ14} for the volume of $k$-dimensional random projections of the $n$-dimensional cube when $n\to\infty$, a result that had previously been obtained by
Kabluchko, Litvak, and Zaporozhets \cite{KLZ15pomi} in the special case $k=1$. Alonso-Guti\'errez, Prochno, and Th\"ale \cite{APT2018} proved a central limit theorem and Berry-Esseen bounds for the Euclidean norm of random orthogonal projections of points chosen uniformly at random from the unit ball of $\ell_p^n$, as $n\to\infty$, and Kabluchko, Prochno, and Th\"ale \cite{KPT2019} obtained a multivariate central limit theorem for the $q$-norm of random vectors chosen uniformly at random in the unit $p$-ball of $\R^n$, which extended the corresponding $1$-dimensional result obtained by Schmuckenschl\"ager \cite{Schmu2001}.

While the results in the previous paragraph describe central limit phenomena for several geometry related quantities, there is considerably less known about the large deviations behavior. Large deviations principles, which appear on the scale of a law of large numbers, have only recently been introduced in geometric functional analysis by Gantert, Kim, and Ramanan \cite{GKR2017}, who obtained a large deviations principle for $1$-dimensional random projections of $\ell_p^n$-balls in $\R^n$, as the space dimension tends to infinity. Subsequent work of Alonso-Guti\'errez, Prochno, and Th\"ale \cite{APT2018} provided a description of the large deviations behavior for the Euclidean norm of projections of $\ell_p^n$-balls to high-dimensional random subspaces (the so-called annealed case), and Kabluchko, Prochno, and Th\"ale \cite{KPT2019} obtained a complete description of the large deviations behavior of $\ell_q$-norms of high-dimensional random vectors that are chosen uniformly at random in an $\ell_p^n$-ball, which can be seen as an asymptotic version of a result of Schechtman and Zinn \cite{SchechtmanZinn}.

The motivation for this manuscript is essentially three-fold and we shall discuss the details
in the following subsections together with our corresponding results. The first is the aim for an extension of the (multivariate) central limit theorems obtained in \cite[Theorem 1.1]{KPT2019} and \cite[Proposition 2.4]{Schmu2001} and the large deviations principles \cite[Theorems 1.2 and 1.3]{KPT2019} to a considerably wider class of distributions on $\ell_p^n$-balls. The second aim is to go between the Gaussian fluctuations described by the central limit theorem and the large deviations and to describe the moderate deviations behavior of the random variables studied there. Moderate deviations are typically non-parametric (in contrast to large deviations) and consider probabilities on scales between those of a law of large numbers and a central limit theorem. These new findings for the moderate scaling therefore complement and refine both the new central limit theorems (Theorem \ref{thm:CLT} and Theorem \ref{thm:CLT general Version2}) as well the new large deviations principle (Theorem \ref{thm:LDP}). For a variety of applications of such results, despite the once presented below, we refer the reader to \cite{KPT2019}.

Before we present our results, let us explain the distributional set-up of this manuscript. As already mentioned, we consider a much more general class of distributions compared to \cite{KPT2019} and \cite{Schmu2001}. Those have been introduced and studied by Barthe, Gu\'edon, Mendelson, and Naor \cite{BartheGuedonEtAl}, and are closely related to the geometry of $\ell_p^n$-balls. This class contains the uniform distribution considered in \cite{APT2018,GKR2017,KPT2019}, the cone probability measure on the $\ell_p^n$-unit ball $\B_p^n :=\{x\in\R^n: \|x\|_p\leq 1\}$ as special cases, and many more (see below). As usual, $\|x\|_p = (|x_1|^p+\ldots +|x_n|^p)^{1/p}$ denotes the $\ell_p$-norm of the vector $x=(x_1,\ldots,x_n)$, and the parameter $p$ satisifes $0< p<\infty$.  For every $n\in\N$, we let $\bW_n$ be any Borel probability measure on $[0,\infty)$, $\bU_{n,p}$ be the uniform distribution, and $\bC_{n,p}$ be the cone probability measure on $\B_p^n$. The distributions we consider are of the form
\begin{equation}\label{eq:DefMeasurePnpW}
\bP_{\bW_n,n,p} := \bW_n(\{0\})\,\bC_{n,p} + H\,\bU_{n,p},
\end{equation}
where the function $H:\B_p^n\to\R$ is given by $H(x)=h(\|x\|_p)$ with
$$
h(r) = {1\over p^{n/p}\Gamma\big(1+{n\over p}\big)}{1\over (1-r^p)^{1+n/p}}\int_0^\infty s^{n/p}e^{-{1\over p}{s r^p (1-r^p)^{-1}}}\,\bW_n(\dint s),\qquad r\in[0,1].
$$
In other words this means that
\begin{align*}
\int_{\B_p^n}f(x)\,\bP_{\bW_n,n,p}(\dint x) &= \bW_n(\{0\})\int_{\SSS_p^{n-1}}f(x)\,\bC_{n,p}(\dint x) + \int_{\B_p^n}f(x)\,H(x)\,\bU_{n,p}(\dint x)\\
&=\bW_n(\{0\})\int_{\SSS_p^{n-1}}f(x)\,\bC_{n,p}(\dint x) + \int_{\B_p^n}f(x)\,h(\|x\|_p)\,\bU_{n,p}(\dint x)
\end{align*}
for all non-negative measurable functions $f:\B_p^n\to\R$, where $\SSS_p^{n-1}=\{x\in\R^n:\|x\|_p=1\}$ denotes the $\ell_p^n$-sphere. The class of measures of the form $\bP_{\bW_n,n,p}$ contains the following important cases, which are of particular interest (see Theorem 1, Theorem 3, Corollary 3, and Corollary 4 in \cite{BartheGuedonEtAl}):
\begin{itemize}
\item[(i)] If $\bW_n$ is the exponential distribution with rate $1/p$ (and mean $p$), then $\bW_n(\{0\})=0$, $H\equiv 1$, and $\bP_{\bW_n,n,p}$ reduces to the uniform distribution $\bU_{n,p}$ on $\B_p^n$.
\item[(ii)] If $\bW_n=\delta_0$ is the Dirac measure concentrated at $0$, then $\bW_n(\{0\})=1$, $H\equiv 0$, and $\bP_{\bW_n,n,p}$ is just the cone probability measure on $\B_p^n$.
\item[(iii)] If $\bW_n={\rm Gamma}(\alpha,1/p)$ is a gamma distribution with shape parameter $\alpha>0$ and rate $1/p$, then $\bP_{\bW_n,n,p}$ is the beta-type probability measure on $\B_p^n$ with Lebesgue density given by
$$
x\mapsto
{\Gamma\big(\alpha+{n\over p}\big)\over\Gamma(\alpha) \big(2 \Gamma\big(1+{1\over p}\big)\big)^n}\,\big(1-\|x\|_p^p\big)^{\alpha-1},
\qquad x\in\B_p^n\,.
$$
In particular, if $\alpha=m/p$ for some $m\in\N$, this is the image of the cone probability measure $\bC_{n+m,p}$ on $\B_p^{n+m}$ under the orthogonal projection onto the first $n$ coordinates. Similarly, if $\alpha=1+m/p$, this distribution arises as the image of the uniform distribution $\bU_{n+m,p}$ on $\B_p^{n+m}$ under the same orthogonal projection.
\end{itemize}
After having discussed the class of distributions we consider, we now turn to our main results.

\begin{rmk}
\rm Note that although for $0<p<1$ the unit balls $\B_p^n$ are not convex, we decided to include them into our analysis, simply because our results are valid in this regime as well. On the other hand, we leave out the case $p=\infty$, since in this case we can only treat the uniform distribution on $\B_\infty^n$ and this was already studied in \cite{KPT2019}.
\end{rmk}

\subsection{Central limit theorems}

The first result in this manuscript is a generalization of the central limit theorems \cite[Theorem 1.1]{KPT2019} and \cite[Proposition 2.4]{Schmu2001} to the broader class of distributions presented above. While the result can in principle be proved in a multivariate form, we prefer to stay in the one-dimensional setting for clarity and for ease of comparison with the moderate and large deviations principles discussed in the next subsections. The theorem below describes the Gaussian fluctuations of the $q$-norm of vectors chosen at random from the balls $\B_p^n$ according to the measures $\bP_{\bW_n,n,p}$. In this paper, we denote by $\overset{\Pro}{\longrightarrow}$ and $\overset{d}{\longrightarrow}$ convergence in probability and in distribution, respectively. Moreover, we put
\begin{equation}\label{eq:Mpq}
M_p(q)
:=
\frac{p^{q/p}}{q+1}\frac{\Gamma(1+\frac{q+1}{p})}{\Gamma(1+{1\over p})}
=
p^{q/p} \frac{\Gamma(\frac{q+1}{p})}{\Gamma({1\over p})}.
\end{equation}
for any $q>0$.

\begin{thmalpha}[Central limit theorem]\label{thm:CLT}
Fix $0<p<\infty$ and $0<q<\infty$.
Let $(\bW_n)_{n\in\N}$ be a sequence of Borel probability measures on $[0,\infty)$. For each $n\in\N$ let $Z_n\in\B_p^n$ be distributed according to $\bP_{\bW_n,n,p}$ and $W_n$ according to $\bW_n$. Assume that
\begin{align}\label{eq:AssumptionsWn}
{W_n\over\sqrt{n}}\,\,\toprob\,\, 0\,.
\end{align}
Then
\begin{align*}
\sqrt{n}\bigg({n^{{1\over p}-{1\over q}}\over M_p(q)^{1/q}}\|Z_n\|_q-1\bigg)\,\,\todistr\,\, N,
\end{align*}
where $N\sim\mathcal{N}(0,\sigma^2)$ is a centered Gaussian random variable with variance
\begin{align*}
\sigma^2 = {1\over q^2}\bigg({\Gamma({1\over p})\Gamma({2q+1\over p})\over\Gamma({q+1\over p})^2}-1\bigg)-{1\over p}.
\end{align*}
\end{thmalpha}

Let us return to the situations (i)--(iii) described above and discuss some special cases of Theorem \ref{thm:CLT}. If for each $n$, $\bW_n=\bW$ for some fixed Borel probability measure $\bW$ on $[0,\infty)$, then assumption \eqref{eq:AssumptionsWn} is clearly satisfied. In particular, taking $\bW$ to be the Dirac measure at zero (recall (ii) above) or the exponential distribution with rate $1/p$ (recall (i) above), we recover the central limit theorem of Schmuckenschl\"ager \cite{Schmu2001}, see also Kabluchko, Prochno, and Th\"ale \cite{KPT2019}. As another example, we fix a sequence of positive real numbers $(a_n)_{n\in\N}$ such that $a_n/\sqrt{n}\to 0$, as $n\to\infty$, and let for each $n\in\N$, $\bW_n=\Gamma(a_n,b)$ be the gamma distribution with shape parameter $a_n$ and some fixed rate $b\in(0,\infty)$. Markov's inequality implies that \eqref{eq:AssumptionsWn} is satisfied in this case, from which the central limit theorem follows. In particular, taking $b=1/p$ we cover the situation discussed under (iii) above.

\begin{rmk}
\rm
In the case when $p=q$, the asymptotic variance $\sigma^2$ vanishes. In this case, Theorem \ref{thm:CLT} just states the distributional convergence of $\sqrt n (\|Z_n\|_p - 1)$ to $0$.
\end{rmk}

\begin{rmk}
\rm Theorem \ref{thm:CLT} should be compared with the (multivariate) central limit theorem for $\|Z_n\|_q$ proved in \cite{ProchnoThaeleTurchiChapter}. The latter is valid under the condition that $\sqrt{n}(1-\|Z_n\|_p)\toprob 0$, as $n\to\infty$. One can in fact show (with some efforts, see the previous remark) that our condition \eqref{eq:AssumptionsWn} implies the one in \cite{ProchnoThaeleTurchiChapter}. However, we prefer to give an alternative and separate argument, since it can be developed further to give a proof of our MDP.
\end{rmk}

In one of our applications we present in Section \ref{sec:Appl} below, a slight generalization of Theorem \ref{thm:CLT} is needed, where we allow  the random variables $W_n$ to converge to a non-trivial limiting distribution after a suitable centering and rescaling by $\sqrt n$.

\begin{thmalpha}[Generalized central limit theorem]\label{thm:CLT general Version2}
Fix $0<p<\infty$ and $0<q<\infty$. Let $(\bW_n)_{n\in\N}$ be a sequence of Borel probability measures on $[0,\infty)$. For each $n\in\N$ let $Z_n\in\B_p^n$ be distributed according to $\bP_{\bW_n,n,p}$ and $W_n$ according to $\bW_n$. Assume that
\begin{align}\label{eq:AssumptionsWnVersion2}
W_n^* := \frac{W_n - \mu_n}{\sqrt{n}} \,\,\todistr\,\, \zeta
\end{align}
with $\zeta\sim\mathcal N(0,\tau^2)$ for some $\tau^2\geq 0$, where $(\mu_n)_{n\in\N}$ is a sequence of non-negative real numbers satisfying $\mu_n /n \to  \mu \in[0,\infty)$, as $n\to\infty$. Then
\begin{align*}
\sqrt{n}\Bigg({n^{{1\over p}-{1\over q}} \frac{(1+\frac{\mu_n}{n})^{1/p}}{ M_p(q)^{1/q}}}\|Z_n\|_q-1\Bigg)\,\,\todistr\,\, \tilde N\,,
\end{align*}
where $\tilde N\sim\mathcal{N}(0,\tilde\sigma^2)$ is a centered Gaussian random variable with variance
$$
\tilde\sigma^2 = {\Gamma({1\over p})\Gamma({2q+1\over p})\over q^2\Gamma({q+1\over p})^2}-{1\over q^2} + \frac{1}{p(1+\mu)^2} - \frac{2}{p(1+\mu)} +\frac{\tau^2}{p^2(1+\mu)^2}.
$$
\end{thmalpha}

We emphasize that Theorem \ref{thm:CLT general Version2} is indeed a generalization of Theorem \ref{thm:CLT}. Namely, if \eqref{eq:AssumptionsWnVersion2} is satisfied with $\mu_n=0$ for all $n\in\N$ and $\tau^2=0$ then $W_n/\sqrt{n}\todistr 0$ and hence $W_n/\sqrt{n}\toprob 0$, as $n\to\infty$, so that \eqref{eq:AssumptionsWn} is satisfied. Moreover, let us briefly mention that Theorem \ref{thm:CLT general Version2} allows us to consider, for example, a gamma distribution $\Gamma(a_n,b)$ for $\bW_n$ with constant rate $b\in(0,\infty)$ and shape parameter $a_n\in(0,\infty)$ satisfying $a_n/n\to a\in[0,\infty)$, as $n\to\infty$. We take advantage of this flexibility in Section \ref{sec:Appl} below.

\subsection{Moderate deviations principle}

We will next describe the moderate deviations. A moderate deviations principle (MDP) is formally nothing else than a large deviations principle (LDP) but with important differences in the behavior of the two principles. For instance, while LDPs provide estimates on the scale of a law of large numbers, MDPs describe the probabilities at scales between a law of large numbers and a distributional limit theorem (like a central limit theorem). Moreover, while the rate function in an LDP depends in a subtle way on the distribution of the underlying random variables, the rate function in an MDP in typical situations is non-parametric and given by the Gaussian one inherited from a central limit theorem.
Let us recall that a sequence $(X_n)_{n\in\N}$ of random vectors in $\R^d$ ($d\in\N$) satisfies an LDP with {\emph speed} $s_n$ and  {\emph `good rate function'} $\rate:\R^d\to[0,\infty]$ if
\begin{equation*}
\begin{split}
-\inf_{x\in A^\circ}\rate(x) &\leq\liminf_{n\to\infty}s_n^{-1}\log\Pro[X_n\in A]\leq\limsup_{n\to\infty}s_n^{-1}\log\Pro[X_n\in A]\leq-\inf_{x\in\overline{A}}\rate(x)
\end{split}
\end{equation*}
for all measurable $A\subseteq\R^d$ ($A^\circ$ being the interior and $\overline A$ the closure of $A$), where $\rate$ is lower semi-continuous and has compact level sets $\{x\in\R^d\,:\, \rate(x) \leq \alpha \}$, $\alpha\in\R$. We say in this paper that a  sequence $(X_n)_{n\in\N}$ satisfies an MDP if the speed sequence $(s_n)_{n\in\N}$ is given by $s_n=b_n\sqrt{n}$ with a positive sequence $(b_n)_{n\in\N}$ satisfying $b_n=\omega(1)$ and $b_n=o(\sqrt{n})$, where for two sequences $(x_n)_{n\in\N}$ and $(y_n)_{n\in\N}$ we use the Landau notation $x_n=o(y_n)$ if $\lim_{n\to\infty} \frac{x_n}{y_n}=0$ and $x_n=\omega(y_n)$ if $\lim_{n\to\infty}|\frac{x_n}{y_n}|=+\infty$.
In our case, the random variables $X_n$ are suitably scaled versions of the $q$-norm of random points in $\B_p^n$.

The following MDP complements both the central limit theorems (Theorem \ref{thm:CLT}, Theorem \ref{thm:CLT general Version2} and also \cite[Theorem 1.1]{KPT2019}) as well as the large deviations result proved in \cite[Theorem 1.2]{KPT2019} and Theorem \ref{thm:LDP} below.

\begin{thmalpha}[Moderate deviations principle]\label{thm:MDP}
Fix $0<p< \infty$ and $0<q<\infty$ with $q<p$. Let $(\bW_n)_{n\in\N}$ be a sequence of Borel probability measures on $[0,\infty)$ and $(b_n)_{n\in\N}$ be a sequence of positive real numbers satisfying $b_n=\omega(1)$ and $b_n=o(\sqrt{n})$.
For each $n\in\N$ let $Z_n\in\B_p^n$ be distributed according to $\bP_{\bW_n,n,p}$. Assume that, for all $\delta>0$,
\begin{align}\label{eq:AssumptionMDP}
\limsup_{n\to\infty}{1\over b_n^2}\log\bW_n\big((\delta b_n\sqrt{n},\infty)\big) =-\infty.
\end{align}
Then the sequence of random variables
\begin{align*}
{\sqrt{n}\over b_n}\bigg({n^{{1/p}-{1/q}}\over M_p(q)^{1/q}}\|Z_n\|_q-1\bigg)
\end{align*}
satisfies an MDP with speed $b_n^2$ and good rate function $\rate(t)={t^2\over 2\sigma^2}$, $t\in\R$, where $\sigma^2$ is the variance from Theorem \ref{thm:CLT}.
\end{thmalpha}

In particular, Theorem \ref{thm:MDP} implies that, for all $t\in\R$,
$$
\lim_{n\to\infty}{1\over b_n^2}\log\Pro\Bigg[{\sqrt{n}\over b_n}\bigg({n^{{1/ p}-{1/q}}\over M_p(q)^{1/q}}\|Z_n\|_q-1\bigg)\geq t\Bigg] = -{t^2\over 2\sigma^2}.
$$

Let us briefly return to the special cases (i)--(iii). Clearly, if $\bW_n$ is the Dirac measure at zero, Assumption \eqref{eq:AssumptionMDP} is satisfied. This covers case (ii) from above. On the other hand, if for each $n\in\N$, $\bW_n=\Gamma(a_n,b)$ is a gamma distribution with shape parameter $a_n\in(0,\infty)$ and rate $b>0$, we can use the MDP for sums of independent random variables (see Lemma \ref{lem:MDPVector} below) to conclude that Assumption \eqref{eq:AssumptionMDP} is satisfied if $a_n=\omega(\sqrt{n}b_n)$. Especially, taking $b=1/p$, this covers cases (i) and (iii).

\begin{rmk}
\rm If $p=q$, then the core term for the MDP that we study in Lemma \ref{lem:MDPcoreterms} below vanishes and therefore, we do not obtain an MDP with a non-trivial rate function.
\end{rmk}

\subsection{Large deviations principle}

The third type of limit theorem we obtain is a large deviations principle. As we shall see in a moment, contrary to the quadratic and non-parametric rate function in the MDP, the LDP is more sensitive to the underlying distribution and displays a significant difference in behavior depending on the parameter $p$ and its relative position with respect to the parameter $q$.

\begin{thmalpha}[Large deviations principle]\label{thm:LDP}
Fix $0<p< \infty$ and $0<q<\infty$ with $p\neq q$. Let $(\bW_n)_{n\in\N}$ be a sequence of Borel probability measures on $[0,\infty)$ and for each $n\in\N$ let $W_n$ be distributed according to $\bW_n$. For each $n\in\N$ let $Z_n\in\B_p^n$ be distributed according to $\bP_{\bW_n,n,p}$.
Then the sequence of random variables $n^{{1/p}-{1/q}}\|Z_n\|_q$ satisfies the following LDPs:
\begin{enumerate}
\item If $q<p$ we assume that the sequence $(W_n/n)_{n\in\N}$ satisfies an LDP with speed $n$ and good rate function $\rate_{\bW}$. Then the LDP is with speed $n$ and good rate function $\rate_{{\bf Z},1}=(\rate_1+\rate_{\bW})\circ F^{-1}$, where $F(x,y,z)=x^{1/q}(y+z)^{-1/p}$ and $\rate_1=\Lambda^*$ is the Legendre-Fenchel transform of the function
\[
\Lambda(t_1,t_2)=\log\int_0^\infty e^{t_1 x^q+(t_2-1/p)x^p}{\dint x\over p^{1/p}\Gamma(1+1/p)}.
\]
\item If $q>p$ we assume that sequence $(W_n/n)_{n\in\N}$ is exponentially equivalent to $0$ in the sense that
\begin{align}\label{eq:AssumptionLDP}
\limsup_{n\to\infty}{1\over n^{p/q}}\log \Pro\bigg[{W_n\over n}>\delta\bigg]=-\infty
\end{align}
for all $\delta>0$. Then the LDP is with speed $n^{p/q}$ and good rate function
\[
\rate_{{\bf Z},2}(x) = \begin{cases}
\frac{1}{p}\big(x^q-M_p(q)\big)^{p/q} & : x\geq M_p(q)^{1/q} \\
+\infty & :\text{otherwise}.
\end{cases}
\]
\end{enumerate}
\end{thmalpha}

We emphasize that while the rate function $\rate_{{\bf Z},2}$ for $q>p$ is universal in the sense that it does not depend on $\rate_\bW$ (provided that $\rate_\bW$ does not vanish in a neighborhood of $1$), this is not the case for the rate function $\rate_{{\bf Z},1}$ for $q<p$, which in a subtle way depends on $\rate_\bW$. As examples we consider the special cases (i) and (ii) above. If for each $n\in\N$, $\bW_n$ is the Dirac measure at zero, the function $\rate_{\bW}$ is given by
$$
\rate_\bW(x) = \begin{cases}
0 &: x= 0\\
+\infty &:x\neq 0.
\end{cases}
$$
Moreover, if $\bW_n$ is the exponential distribution with parameter $1/p$ for each $n\in\N$, then
$$
\rate_{\bW}(x) = \begin{cases}
+\infty &: x < 0\\
{x\over p} &: x\geq 0.
\end{cases}
$$

\begin{rmk}\rm
If $p=q$, then the LDP of Theorem \ref{thm:LDP} (1) remains valid in a modified form. In fact, it still holds with speed $n$, but the rate function is then given by $(\widetilde{\Lambda}^*+\rate_\bW)\circ\widetilde{F}^{-1}$, where $\widetilde{\Lambda}^*$ is the Legendre-Fenchel transform of
\[
\widetilde{\Lambda}(t) = \log \int_0^\infty e^{(t-1/p)x^{p}}\frac{\dint x}{p^{1/p}\Gamma(1+1/p)}
\]
and $\widetilde{F}$ is the function $\widetilde{F}:(t_1,t_2)\mapsto t_1^{1/p}/(t_1+t_2)^{1/p}$.
\end{rmk}

\subsection{Structure} The remaining parts of this text are structured as follows. Two applications of our results to random and non-random projections of $\ell_p^n$-balls are discussed in Section \ref{sec:Appl}. In Section \ref{sec:Prelim} we rephrase some preliminary results, which are used in proofs of Theorems \ref{thm:CLT}, \ref{thm:CLT general Version2}, \ref{thm:MDP}, and \ref{thm:LDP}. The latter are contained in Section \ref{sec:Proofs}. More precisely, we develop a crucial probabilistic representation for the involved random variables in Section \ref{subsec:ProbRep} and then prove Theorem \ref{thm:CLT} in Section \ref{subsec:ProofCLT}, Theorem \ref{thm:CLT general Version2} in Section \ref{subsec:ProofGeneralCLT}, Theorem \ref{thm:MDP} in Section \ref{subsec:ProofMDP}, and Theorem \ref{thm:LDP} in Section \ref{subsec:ProofLDP}.

\section{Application to projections of $\ell_p^n$-balls}\label{sec:Appl}

\subsection{Random versus non-random subspaces}

Projections of $\ell_p^n$-balls to lower-dimensional subspaces were subject of a number of studies, see, e.g., \cite{APT2018,APTCLT,GKR2017,KimRamanan,Meckes,Meckes2}. In these works two different set-ups were studied, one in which the subspace one projects onto is random, and another one, in which the choice of the subspace is deterministic (for an extensive comparison of both situations for one-dimensional projections we refer the reader to \cite{GKRCramer,GKR2017}). We shall use the limit theory for the general distributions $\bP_{\bW_n,n,p}$ on $\ell_p^n$-balls presented in the previous section to compare both approaches. We start by recalling the framework for projections onto random subspaces taken from \cite{APT2018,APTCLT}. We let $(k_n)_{n\in\N}$ be a sequence of integers satisfying $k_n\in\{1,\ldots,n\}$ and $k_n/n\to\lambda\in[0,1]$, as $n\to\infty$. Moreover, for each $n\in\N$, let $X_n$ be uniformly distributed on $\B_p^n$ and let $E_n$ be a uniformly distributed $k_n$-dimensional random subspace (where the uniform distribution refers to the Haar probability measure on the Grassmannian of all $k_n$-dimensional linear subspaces in $\R^n$). We assume that the two sequences $(X_n)_{n\in\N}$ and $(E_n)_{n\in\N}$ are independent. Moreover, we denote by $P_{E_n}X_n$ the orthogonal projection of $X_n$ onto $E_n$. The quantity studied in \cite{APT2018,APTCLT} is the Euclidean norm of the projection of the random vector $X_n$ onto the random subspace $E_n$, i.e., $\|P_{E_n}X_n\|_2$.

We first rephrase the central limit theorem \cite[Theorem 1.1]{APTCLT}. It says that if $k_n\to\infty$, as $n\to\infty$, then
\begin{align}\label{eq:CLTfromBernoulliPaper}
{n^{{1/p}}\over\sqrt{M_p(2)}}\|P_{E_n}X_n\|_2-\sqrt{k_n}\todistr N,
\end{align}
where $N$ is a centered Gaussian random variable with variance
$$
\sigma^2(p,\lambda)={\lambda\over 4}{\Gamma({1\over p})\Gamma({5\over p})\over\Gamma({3\over p})^2}-\lambda\Big({3\over 4}+{1\over p}\Big)+{1\over 2}.
$$
Observe that taking $\lambda=1$ the constant $\sigma^2(p,1)$ coincides with $\sigma^2$ from Theorem \ref{thm:CLT} if we take $q=2$ there.

Next, we recall the LDP for the same quantities from \cite[Theorem 1.2]{APT2018} (for simplicity we restrict ourselves to the case $p<2$, since only in this case an explicit form of the rate function is available). Using the same notation as before, it says that for any $p\in[1,2)$ the sequence of random variables $n^{{1\over p}-{1\over 2}}\|P_{E_n}X_n\|_2$ satisfies an LDP with speed $n^{p/2}$ and good rate function
\begin{align}\label{eq:RateFunctionFromAPT}
\rate(y) = \begin{cases}
{1\over p}\Big({y^2\over\lambda}-M_p(2)\Big)^{p/2} &: y\geq\sqrt{\lambda M_p(2)}\\
+\infty &: \text{otherwise},
\end{cases}
\end{align}
whenever $\lambda:=\lim\limits_{n\to\infty}{k_n\over n}\in(0,1]$.

The projections onto random subspaces as just described can be compared with projections onto sequences of deterministic subspaces. In fact, our distributional framework allows to deal with projections onto coordinate subspaces. Namely, let the sequence $(k_n)_{n\in\N}$ be as above and let, for each $n\in\N$, $X_n$ be uniformly distributed in the $n$-dimensional $\ell_p^n$-ball $\B_p^n$ with $0<p<\infty$. We denote by $\Pi_{k_n}X_n$ the orthogonal projection of $X_n$ onto the first $k_n$ coordinates. Thus, $\Pi_{k_n}$ is the projection from $\R^n$ to $\{x=(x_1,\ldots,x_n)\in\R^n:x_i=0\text{ for }i>k_n\}$, which in turn can be identified with $\R^{k_n}$.
\begin{thmalpha}[Central limit theorem for deterministic projections]\label{thm:CLT_det_proj}
Assume that $k_n/ n\to\lambda\in(0,1]$, as $n\to\infty$. Then,
$$
{n^{1/p}\over\sqrt{M_p(2)}}\|\Pi_{k_n}X_n\|_2-\sqrt{k_n} \todistr \tilde N,
$$
where $\tilde N$ is a centered Gaussian random variable with variance
$$
\tilde\sigma^2(p,\lambda) = \frac 14 \bigg ({\Gamma({1\over p})\Gamma({5\over p})\over \Gamma({3\over p})^2}-1\bigg) - \frac{\lambda}{p}.
$$
\end{thmalpha}
\begin{proof}
Recalling the special case (iii) for $\bP_{\bW_{k_n},k_n,p}$ from the previous section, we see that the projected random vector $\Pi_{k_n}X_n$ has distribution $\bP_{\bW_{k_n},k_n,p}$ on $\B_p^{k_n}$, where $\bW_{k_n}=\Gamma({n-k_n\over p}+1,{1\over p})$ is a gamma distribution with shape parameter ${n-k_n\over p}+1$ and rate $1/p$.
We are going to apply the central limit theorem to the gamma distribution with the aim of verifying condition~\eqref{eq:AssumptionsWnVersion2} of Theorem~\ref{thm:CLT general Version2}. Keeping in mind that $k_n$ is now the dimension parameter of the projection, we define
$$
\mu_{k_n}:=\E W_{k_n} = n-k_n+p
\qquad\text{and}\qquad
\mu:=\lim_{n\to\infty}{\mu_{k_n}\over k_n} = {1-\lambda\over\lambda}\geq 0\,.
$$
In addition, we have that
$$
{\Var\, W_{k_n}\over k_n} = p\Big({n\over k_n}-1\Big) + {p^2\over k_n} \to p\,{1-\lambda\over\lambda}=:\tau^2\geq 0\,,
$$
as $n\to\infty$. Assume, for a moment, that $n-k_n\to\infty$. Then, even though $\tau^2$ can vanish, we have $\Var\, W_{k_n} \to\infty$.  Under these circumstances, the central limit theorem is applicable to the gamma distribution and yields that
\begin{equation}\label{eq:W_k_n_*_CLT}
W_{k_n}^* := \frac{W_{k_n} - \mu_{k_n}}{\sqrt{k_n}} \todistr \zeta \sim \mathcal{N}(0,\tau^2)\,.
\end{equation}
On the other hand, if $n-k_n$ stays bounded, then $\mu_{k_n}$ stays bounded, hence the sequence $(W_{k_n})_{n\in\N}$ is tight, and since $k_n\to\infty$ (recall that $\lambda\neq 0$), we conclude that~\eqref{eq:W_k_n_*_CLT} still holds with $\tau^2=0$. Summarizing, we conclude that~\eqref{eq:W_k_n_*_CLT} always holds under the assumptions of the theorem. Indeed assume that~\eqref{eq:W_k_n_*_CLT} is violated. Since the sequence $(W_{k_n}^* )_{n\in\N}$ has uniformly bounded variances, we could pass to a subsequence for which $W_{k_n}^*$ converges weakly to some distribution different from $\mathcal N(0,\tau^2)$. Passing one more time to a subsequence, we could assume that either $n-k_n\to\infty$ or $n-k_n$ is bounded. However, as was explained above, this would lead to a contradiction.

We can thus apply Theorem \ref{thm:CLT general Version2} with $q=2$ and dimension parameter $k_n$ instead of $n$ to conclude that
$$
\sqrt{k_n}\Bigg(k_n^{{1\over p}-{1\over 2}}{\big(1+{\mu_{k_n}\over k_n}\big)^{1/p}\over\sqrt{M_p(2)}}\|\Pi_{k_n}X_n\|_2-1\Bigg)
\todistr \tilde N,
$$
where $\tilde N$ is a centered normal random variable with variance
\begin{align*}
\tilde \sigma^2(p,\lambda)
&=
\frac 14 \bigg ({\Gamma({1\over p})\Gamma({5\over p})\over \Gamma({3\over p})^2}-1\bigg) + \frac{1}{p(1+\mu)^2} - \frac{2}{p(1+\mu)} +\frac{\tau^2}{p^2(1+\mu)^2}\\
&=
\frac 14 \bigg ({\Gamma({1\over p})\Gamma({5\over p})\over \Gamma({3\over p})^2}-1\bigg) + \frac{\lambda^2}{p} - \frac{2\lambda}{p} + \frac{(1-\lambda)\lambda}{p}\\
&=\frac 14 \bigg ({\Gamma({1\over p})\Gamma({5\over p})\over \Gamma({3\over p})^2}-1\bigg)-{\lambda\over p}.
\end{align*}
After recalling that $\mu_{k_n}= n-k_n+p$, this can be written in the form
\begin{equation}\label{eq:CLT mit n+p}
 {(n+p)^{1/p}\over\sqrt{M_p(2)}}\|\Pi_{k_n}X_n\|_2-\sqrt{k_n} \todistr \tilde N.
\end{equation}
To complete the proof, we need to replace the factor $(n+p)^{1/p}$ by $n^{1/p}$. That this is always possible can be seen as follows. For $n\in\N$ we define
$$
a_n := {(n+p)^{1/p}\over\sqrt{M_p(2)}}\,,\qquad a_n':={n^{1/p}\over\sqrt{M_p(2)}}\,,\qquad b_n:=\sqrt{k_n}\qquad \text{and}\qquad \xi_n:=\|\Pi_{k_n}X_n\|_2.
$$
Then \eqref{eq:CLT mit n+p} reads as $a_n\xi_n-b_n\todistr \tilde N$, and our aim is to show that the same is true with $a_n$ replaced by $a_n'$. To this end we write
$$
a_n'\xi_n-b_n = (a_n\xi_n-b_n)\,{a_n'\over a_n} - b_n\Big(1-{a_n'\over a_n}\Big).
$$
Since $a_n'/a_n\to 1$, as $n\to\infty$, the first term converges in distribution to $\tilde N$ by Slutsky's theorem, and it remains to prove that $b_n\big(1-{a_n'\over a_n}\big)\to 0$. This is done as follows:
$$
b_n\bigg(1-{a_n'\over a_n}\bigg)
=
\sqrt{k_n}\bigg(1-{n^{1/p}\over(n+p)^{1/p}}\bigg)
=
\sqrt{k_n}\big(1-(1+p/n)^{-1/p}\big)
=
O\bigg(\frac{\sqrt{k_n}}n\bigg)
\to 0.
$$
Summarizing, we have shown that \eqref{eq:CLT mit n+p} is in fact equivalent to
$$
{n^{1/p}\over\sqrt{M_p(2)}}\|\Pi_{k_n}X_n\|_2-\sqrt{k_n} \todistr \tilde N,
$$
thus completing the proof.
\end{proof}

Theorem~\ref{thm:CLT_det_proj}, together with~\eqref{eq:CLTfromBernoulliPaper},  leads us to the remarkable observation that we have the same central limit behavior regardless of whether we project onto uniform random subspaces of dimensions $k_n$ or onto deterministic coordinate subspaces of the same dimension, provided their dimension is sufficiently large, i.e., if $k_n/n\to 1$ as $n\to\infty$. Indeed, the centering in both results is the same, and it is easy to check that $\sigma^2(p,1)= \tilde \sigma^2(p,1)$.  On the other hand, if $k_n/n \to \lambda \in (0,1)$, we still have a central limit theorem for the (suitably centered and rescaled) quantities $\|P_{E_n}X_n\|_2$ and $\|\Pi_{k_n}X_n\|_2$, with the same centering, but this time with different limiting variances $\sigma^2(p,\lambda)$ and $\tilde\sigma^2(p,\lambda)$, respectively.

\medbreak

A similar comparison as for the central limit theorem can be made on the large deviations scale. We restrict ourselves to the case $1\leq p<2$ and $k_n/n\to 1$, that is $\lambda=1$. We are interested in large deviations of $\|\Pi_{k_n} X_n\|_2$, which is distributed as the $2$-norm of a random vector with the probability law $\bP_{\bW_{k_n},k_n,p}$ on $\B_p^{k_n}$, where $\bW_{k_n}$ is the gamma distribution $\Gamma({n-k_n\over p}+1,{1\over p})$, as above. Let us check that the sequence of random variables $W_{k_n}/k_n$ with $W_{k_n}$ having distribution $\bW_{k_n}$ is exponentially equivalent to $0$ in the sense of \eqref{eq:AssumptionLDP}. Fix some $\delta>0$. Since $n-k_n=o(n)$, the convolution property of the gamma distribution in its shape parameter entails that, for large $n$, the random variable $W_{k_n}$ is stochastically dominated by a sum $S_{[\delta n/4]}$ of $[\delta n/4]$ i.i.d.\ $\Gamma(1,{1\over p})$-distributed random variables. Note that $\E S_{[\delta n/4]} = p[\delta n/4]$. Moreover, again for $n$ sufficiently large, $\frac{k_n}{n}>\frac{p}{2}$. We deduce from this and Cram\'er's theorem (see Lemma \ref{lem:cramer} below) that for large $n\in\N$
$$
\Pro\big[{W_{k_n}/ k_n}>\delta\big] = \Pro\big[W_{k_n}>n\delta k_n/n\big] \leq \Pro\big[{W_{k_n}}> p\delta n/2\big] \leq \Pro \big[S_{[\delta n/4]} > p\delta n/2 \big] \leq e^{-c n} \leq e^{-c k_n},
$$
where $c=c(\delta,p)\in(0,\infty)$ is some constant depending on $\delta$ and $p$, but since $p\in[1,2)$ the dependence on $p$ can be omitted. Note that the above argument would fail if $k_n/n\to \lambda<1$.  Thus,
$$
\limsup_{n\to\infty}{1\over k_n^{p/2}}\log \Pro\Big[{W_{k_n}\over k_n}>\delta\Big] \leq \limsup_{n\to\infty}-{c\,k_n\over k_n^{p/2}} = -\infty,
$$
since $p<2$. In this case,  Theorem \ref{thm:LDP} can be applied with $q=2$ and we obtain an LDP for $k_n^{1/p-1/2}\|\Pi_{k_n}X_n\|_2$ with speed $k_n^{p/2}$ and the rate function given in Theorem \ref{thm:LDP}. Since $k_n/n\to 1$, we conclude that $n^{1/p-1/2}\|\Pi_{k_n}X_n\|_2$ satisfies an LDP with speed $n^{p/2}$ and the same rate function $\rate$ as in~\eqref{eq:RateFunctionFromAPT} with $\lambda =1$ there.  Again, this shows that the same large deviations behavior is present regardless of whether we project onto uniform random subspaces of dimensions $k_n$ or onto deterministic coordinate subspaces of the same dimension, again provided their dimension is sufficiently large in the sense that $k_n/n\to 1$, as $n\to\infty$.

\subsection{$1$-dimensional random projections of $\ell_p^n$-balls}

In this section we present another application of our main results demonstrating the advantage of studying the more general distributions $\bP_{\bW,n,p}$  on the $\ell_p^n$-balls. In \cite[Corollary 2.6]{KPT2019}, we proved a generalization to $\ell_p^n$-balls of a central limit theorem obtained by Paouris, Pivovarov, and Zinn \cite[p. 703]{PPZ14} and Kabluchko, Litvak, and Zaporozhets \cite[Theorem 3.6]{KLZ15pomi} for the width of orthogonal projections of the $n$-dimensional cube $\B_\infty^n$ onto a uniformly distributed random direction. For $1 < q <\infty$ with $q\neq 2$ and a random vector chosen from $\SSS^{n-1}$ with respect to the cone probability measure (which in this case coincides with the normalized spherical Lebesgue measure), it was shown in \cite{KPT2019} that, as $n\to\infty$,
\[
\frac{n^{1/q} \vol_1(P_\theta\B_q^n)}{2M_2(q^*)^{1/q^*}} - \sqrt{n} \stackrel{d}{\longrightarrow} N
\]
where $N$ is a centered Gaussian random variable with variance
\begin{align}\label{eq: sigma_q^2}
\sigma^2(q) & = \frac{1}{(q^*)^2}\bigg(\frac{\sqrt{\pi}\, \Gamma(\frac{2q^*+1}{2})}{\Gamma(\frac{q^*+1}{2})^2}-1\bigg) - \frac{1}{2}.
\end{align}
Here, $q^*$ denotes the H\"older conjugate of $q$ satisfying $\frac{1}{q}+\frac{1}{q^*}=1$, $P_\theta$ denotes the orthogonal projection onto the line spanned by $\theta$, and
\begin{equation}\label{eq:representation width}
\vol_1(P_\theta\B_q^n) = 2 \sup_{x\in\B_q^n}|\langle x,\theta \rangle| = 2 \|\theta\|_{q^*}.
\end{equation}
While the argument to obtain this central limit theorem had to be extracted from the proof of the main result \cite[Theorem 1.1]{KPT2019}, it is in our set-up a direct consequence of Theorem \ref{thm:CLT}, since we study more general distributions for which the choice $\bW_n=\delta_0$ and $p=2$ yields that $\bP_{\bW_n,n,2}$ is just the cone probability measure on $\B_2^n$. More precisely, to obtain the central limit theorem above, we use the representation \eqref{eq:representation width} and apply Theorem \ref{thm:CLT} with the choice $\bW_n=\delta_0$, $p=2$, $q$ replaced by $q^*$, and take $Z_n=\theta$.

Beyond the Gaussian fluctuations just described, our results in Theorems \ref{thm:MDP} and \ref{thm:LDP} concerning moderate and large deviations allow us to deduce the complementing MDPs and LDPs for the length of the orthogonal projection of $\B_q^n$ onto a random direction as well. We start with the description of the moderate deviations behaviour. Using Theorem \ref{thm:MDP} with the choice $p=2$, $\bW_n=\delta_0$, and $q$ replaced by $q^*$ with $q^*<p$, we obtain that the sequence of random variables
\[
{n^{1/q}\over b_n}\frac{\vol_1(P_\theta\B_q^n)}{2M_2(q^*)^{1/q^*}} - {\sqrt{n}\over b_n}
\]
satisfies an MDP with speed $b_n^2$ and good rate function $\rate(t) = t^2/(2\sigma^2(q))$, where $(b_n)_{n\in\N}$ is a sequence of positive real numbers satisfying $b_n=\omega(1)$ and $b_n=o(\sqrt{n})$, and the constant $\sigma^2(q)$ is as in \eqref{eq: sigma_q^2}.

The large deviations are obtained similarly.  Using Theorem \ref{thm:MDP} with the choice $p=2$, $\bW_n=\delta_0$, and $q$ replaced by $q^*$ with $q^*>p$ (we restrict ourselves to this case, since only in this case we have a closed form expression for the rate function), we obtain that the sequence of random variables
\[
n^{{1\over q}-{1\over 2}}{\vol_1(P_\theta\B_q^n)\over 2}
\]
satisfies an LDP with speed $n^{2/q^*}=n^{2-2/q}$ and good rate function
\begin{align*}
\rate(x) &= \begin{cases}
\frac{1}{2}\big(x^{q^*}-M_2(q^*)\big)^{2/q^*} & : x\geq M_2(q^*)^{1/q^*} \\
+\infty & :\text{otherwise}
\end{cases}\\
&= \begin{cases}
{1\over 2}\big(x^{q/(q-1)}-M_2(q^*)\big)^{2-2/q} &: x\geq M_2(q^*)^{1-1/q} \\
+\infty & :\text{otherwise}.
\end{cases}
\end{align*}
Finally, we mention that the constant $M_2(q^*)^{1/q^*}$ can be explicitly expressed as $$\sqrt{2\pi^{1-q\over q}}\,\Gamma\Big({2q-1\over 2q-2}\Big)^{1-1/q}$$ in terms of the parameter $q$.

\section{Preliminaries}\label{sec:Prelim}

In this section we briefly present some background material used throughout the rest of this text. For convenience of the reader, we split this into different subsections that may be skipped depending on the reader's background.

\subsection{Generalized Gaussian random variables}

Let us denote, for $0<p<\infty$, by $(Y_i)_{i\in\N}$ a sequence of independent copies of a $p$-generalized Gaussian random variable with Lebesgue density
$$
f_p(x)=c_p^{-1}e^{-\frac{|x|^p}{p}}  ,\quad x\in\R,
$$
where the normalization constant $c_p$ is given by $c_p:=2p^{1/p}\Gamma(1+\frac{1}{p})$. Next, recall the definition of the constant $M_p(q)$ from \eqref{eq:Mpq}. It can be used to express first- and second-order moments of $p$-generalized Gaussian random variables as follows. Namely, for $q,r,s>0$ we have that
\begin{align}\label{eq:Moments}
\E|Y_1|^q = M_p(q)\qquad\text{and}\qquad\Cov(|Y_1|^r,|Y_1|^s)=M_p(r+s)-M_p(r)M_p(s),
\end{align}
see \cite[Lemma 3.1]{APTCLT}. Note that $M_p(p)=1$.

The family of $p$-generalized Gaussian random variables can be used to describe a probabilistic interpretation of the distributions  $\bP_{\bW_n,n,p}$ that were defined in the introduction. This interpretation is one of the key devices in the proofs of Theorems \ref{thm:CLT}, \ref{thm:MDP}, and \ref{thm:LDP}.

\begin{lemma}[Probabilistic interpretation, Theorem 3 in \cite{BartheGuedonEtAl}]\label{lem:prob representation}
Let $0< p< \infty$, $Y^{(n)}=(Y_1,\dots,Y_n)$ be a random vector of independent and $p$-generalized coordinates, and assume that $W_n$ is a non-negative random variable with distribution ${\bf W}_n$, which is independent of $Y^{(n)}$. Then the random vector
\[
\frac{Y^{(n)}}{(\|Y^{(n)}\|_p^p + W_n)^{1/p}}
\]
is distributed according to the measure $\bP_{\bW_n,n,p}$.
\end{lemma}

\subsection{Moderate and large deviations}

Let $(X_n)_{n\in\N}$ be a sequence of random vectors on some probability space $(\Omega,\mathcal A,\Pro)$ taking values in a Hausdorff topological space $\XX$. Further, let $(s_n)_{n\in\N}$ be an increasing sequence of real numbers and $\rate:\XX\to[0,\infty]$ be a lower semi-continuous function with compact level sets $\{x\in\R^d\,:\, \rate(x) \leq \alpha \}$ for all $\alpha\in\R$. One says that $(X_n)_{n\in\N}$ satisfies a large deviations principle (LDP) on $\XX$ with speed $s_n$ and good rate function $\rate$, provided that
\begin{equation*}
\begin{split}
-\inf_{x\in A^\circ}\rate(x) &\leq\liminf_{n\to\infty}s_n^{-1}\log\Pro[X_n\in A]\leq\limsup_{n\to\infty}s_n^{-1}\log\Pro[X_n\in A]\leq-\inf_{x\in\overline{A}}\rate(x)
\end{split}
\end{equation*}
for all Borel sets $A\subseteq\XX$, where $A^\circ$ denotes the interior and $\overline A$ the closure of $A$. As already discussed in the introduction, a moderate deviations principle (MDP) is formally the same as an LDP, but on a different rage of scales.

We shall now present a few basic results from large deviations theory which are needed below. Assume that a sequence $(X_n)_{n\in\N}$ of random variables satisfies an LDP with speed $s_n$ and rate function $I$. Suppose now that $(Y_n)_{n\in\N}$ is a sequence of random variables that are `close' to the ones from the first sequence. The next result provides conditions under which in such a situation an LDP from the first can be transferred to the second sequence.

\begin{lemma}[Exponential equivalence, Theorem 4.2.13 in \cite{DZ}]\label{prop:exponentially equivalent}
Let $(X_n)_{n\in\N}$ and $(Y_n)_{n\in\N}$ be two sequence of $\R^d$-valued random vectors and assume that $(X_n)_{n\in\N}$ satisfies an LDP on $\R^d$ with speed $s_n$ and rate function $\rate$. Further, suppose that the two sequences $(X_n)_{n\in\N}$ and $(Y_n)_{n\in\N}$ are exponentially equivalent, which is to say that
$$
\limsup_{n\to\infty}s_n^{-1}\log\Pro[\|X_n-Y_n\|_2>\delta] = -\infty
$$
for any $\delta>0$. Then $(Y_n)_{n\in\N}$ satisfies an LDP on $\R^d$ with the same speed and the same rate function.
\end{lemma}

Next, we recall what is known as Cram\'er's theorem. It provides an LDP for sequences of independent and identically distributed random variables.

\begin{lemma}[Cram\'er's theorem, Theorem 2.2.3 in \cite{DZ}]\label{lem:cramer}
Let $(X_n)_{n\in\N}$ be a sequence of i.\,i.\,d.\,random variables. Assume that $\E e^{\lambda X_1}<\infty$ for all $|\lambda|<\lambda_0$ for some $\lambda_0>0$. Then the sequence of random variables ${1\over n}\sum_{i=1}^n X_i$ satisfies an LDP on $\R$ with speed $n$ and good rate function $\rate(x)=\sup\big\{\lambda x-\log\E e^{\lambda X_1}:\lambda\in\R\big\}$, i.e., $\rate$ is the Legendre-Fenchel transform of the log-moment generating function $\log\E e^{\lambda X_1}$.
\end{lemma}

Let $d_1,d_2\in\N$ and suppose that $(X_n)_{n\in\N}$ is a sequence of $\R^{d_1}$-valued random vectors and that $(Y_n)_{n\in\N}$ is a sequence of $\R^{d_2}$-random vectors. We assume that both sequences satisfy LDPs with the same speed. The next result, taken from \cite[Proposition 2.4]{APT2018}, yields that also the sequence of $\R^{d_1+d_2}$-valued random vectors $(X_n,Y_n)$ satisfies an LDP and provides the form of the rate function.

\begin{lemma}\label{JointRateFunction}
Assume that $(X_n)_{n\in\N}$ satisfies an LDP on $\R^{d_1}$ with speed $s_n$ and good rate function $\rate_{\bf X}$ and that $(Y_n)_{n\in\N}$ satisfies an LDP on $\R^{d_2}$ with speed $s_n$ and good rate function $\rate_{\bf Y}$. Then, if $X_n$ and $Y_n$ are independent for each $n\in\N$, the sequence of random vectors $(X_n,Y_n)$ satisfies an LDP on $\R^{d_1+d_2}$ with speed $s_n$ and good rate function $\rate$ given by $\rate(x):=\rate_{\bX}(x_1)+\rate_{\bY}(x_2)$, $x=(x_1,x_2)\in\R^{d_1}\times\R^{d_2}$.
\end{lemma}

Finally, we consider the possibility to transport a large deviations principle to another one by means of a continuous function, a result which is known as the so-called contraction principle.

\begin{lemma}[Contraction principle, Theorem 4.2.1 in \cite{DZ}]\label{lem:contraction principle}
Let $\XX$ and $\YY$ be two Hausdorff topological space and let
let $F:\XX\to\YY$ be a continuous function.
Further, let $(X_n)_{n\in\N}$ be a sequence of $\XX$-valued random elements that satisfies an LDP with speed $s_n$ and good rate function $\rate_\bX$. Then the sequence $(F(X_n))_{n\in\N}$ of $\YY$-valued random elements satisfies an LDP with the same speed and with good rate function $\rate = \rate_\bX\circ F^{-1}$, i.e.,
$$
\rate(y):=\inf\{\rate_\bX (x): x\in \XX, F(x)=y\},
\quad
y\in \YY,
$$
with the convention that $\rate(y)=+\infty$ if $F^{-1}(\{y\})=\varnothing$.
\end{lemma}

As explained before, a moderate deviations principle is formally nothing else than a large deviations principle and describes (in our set-up) the deviation probabilities at scales between a law of large numbers and a central limit theorem.
An important tool for us will be the following MDP for sums of independent and identically distributed random vectors.

\begin{lemma}[MDP for sums of random vectors, Theorem 3.7.1 in \cite{DZ}]\label{lem:MDPVector}
Let $(X_n)_{n\in\N}$ be a sequence of independent and identically distributed random vectors in $\R^d$ and let $(s_n)_{n\in\N}$ be sequence of positive real numbers such that $s_n=\omega(\sqrt{n})$ and $s_n=o(n)$. We assume that $X_1$ is centered, its covariance matrix $\bC=\Cov(X_1)$ is invertible, and $\log\E\, e^{\langle \lambda,X_1\rangle}<\infty$ for all $\lambda$ in a ball around the origin having positive radius. Then the sequence of random vectors $\frac{1}{s_n}\sum_{i=1}^nX_i$, $n\in\N$, satisfies an LDP with speed $s_n^2/n$ (i.e.,\ an MDP) and good rate function $\rate(x)={1\over 2}\langle x,\bC^{-1}x\rangle$, $x\in\R^d$.
\end{lemma}

\begin{rmk}\label{rem:MDPExMom}
\rm There exist versions of Lemma \ref{lem:MDPVector} under less restrictive assumptions on the (exponential) moments of the involved random vectors, see \cite{ArconesMDP}, for example. However, such results do not lead to simplifications or improvements in our situation.
\end{rmk}

\section{Proof of the main results}\label{sec:Proofs}

\subsection{A probabilistic representation for the $q$-norm}\label{subsec:ProbRep}

In a first step we develop a probabilistic representation for the random variables $\|Z_n\|_q$, which will turn out to be useful for both, the proof of the central limit theorems and the moderate deviations principle. In what follows we let $Y_1,Y_2,\ldots$ be a sequence of independent $p$-generalized Gaussian random variables and define, for each $n\in\N$,
\begin{align*}
S_n^{(1)} := {1\over\sqrt{n}}\sum_{i=1}^n\big(|Y_i|^q-M_p(q)\big)\quad\text{and}\quad S_n^{(2)} := {1\over\sqrt{n}}\sum_{i=1}^n\big(|Y_i|^p-1\big)\,,
\end{align*}
where $0<p,q<\infty$.

\begin{lemma}[Probabilistic interpretation]\label{lem:ProbRep}

Fix $0<p<\infty$, $0<q<\infty$ and $n\in\N$. Let $\bW_n$ be a Borel probability measure on $[0,\infty)$. Let $Z_n\in\B_p^n$ be distributed according to $\bP_{\bW_n,n,p}$ and $W_n$ be distributed according to $\bW_n$ and independent of $Y_1,Y_2,\ldots$. Then
\begin{align*}
\|Z_n\|_q \overset{d}{=} n^{{1\over q}-{1\over p}}M_p(q)^{1/q}\bigg[1 + {S_n^{(1)}\over qM_p(q)\,\sqrt{n}} - {S_n^{(2)}\over p\,\sqrt{n}} - {W_n\over p\,n}+\Psi_p\Big({S_n^{(1)}\over\sqrt{n}},{S_n^{(2)}\over\sqrt{n}},{W_n\over n}\Big)\bigg],
\end{align*}
where $\Psi_p:\R^3\to\R$ is such that, for some $M,\delta>0$, we have $|\Psi_p(x,y,z)|\leq M\|(x,y,z)\|_2^2$ whenever $\|(x,y,z)\|_2^2<\delta$.
\end{lemma}
\begin{proof}
We first observe that as a consequence of Lemma \ref{lem:prob representation} the random vector $Z_n$ has the probabilistic representation
\begin{align*}
Z_n \overset{d}{=} {Y^{(n)}\over(\|Y^{(n)}\|_p^p+W_n)^{1/p}},
\end{align*}
where $Y^{(n)}=(Y_1,\ldots,Y_n)$ is a vector of independent $p$-generalized Gaussian random variables and $W_n$ is a random variable with distribution $\bW_n$, which is independent of $Y^{(n)}$. Thus
\begin{align*}
\|Z_n\|_q \stackrel{d}{=} {\|Y^{(n)}\|_q\over(\|Y^{(n)}\|_p^p+W_n)^{1/p}}.
\end{align*}
Recalling the definitions of the random variables $S_n^{(1)}$ and $S_n^{(2)}$, we can rewrite the last expression as
\begin{align}\label{eq:ProbRepEq1}
\|Z_n\|_q \stackrel{d}{=} {\big(\sqrt{n}S_n^{(1)}+nM_p(q)\big)^{1/q}\over\big(\sqrt{n}S_n^{(2)}+n+W_n\big)^{1/p}} = n^{{1\over q}-{1\over p}}M_p(q)^{1/q}{\big(1+{S_n^{(1)}\over\sqrt{n}M_p(q)}\big)^{1/q}\over\big(1+{S_n^{(2)}\over\sqrt{n}}+{W_n\over n}\big)^{1/p}}.
\end{align}
Next, we define the function
\begin{align*}
F:D_F\subset\R^3\to\R,\quad (x,y,z)\mapsto{\big(1+{x\over M_p(q)}\big)^{1/q}\over(1+y+z)^{1/p}},
\end{align*}
where $D_F$ stands for the domain of $F$. Clearly, some open neighborhood of $(0,0,0)$ is contained in $D_F$, and a Taylor expansion of $F$ around $(0,0,0)$ shows that for all $(x,y,z)\in D_F$,
\begin{align*}
F(x,y,z) = 1+{x\over qM_p(q)}-{y\over p}-{z\over p}+\Psi_p(x,y,z),
\end{align*}
where the function $\Psi_p:D_F\to\R$ is such that, for some $M,\delta>0$, we have $|\Psi_p(x,y,z)|\leq M\|(x,y,z)\|_2^2$ whenever $\|(x,y,z)\|_2^2<\delta$. Combining this with the representation \eqref{eq:ProbRepEq1} for $\|Z_n\|_q$ proves the claim.
\end{proof}

\subsection{Proof of the central limit theorem (Theorem \ref{thm:CLT})}\label{subsec:ProofCLT}

For each $n\in\N$ let us define the random variable
\begin{align}\label{eq:DefV_n}
V_n := \sqrt{n}\bigg({n^{{1\over p}-{1\over q}}\over M_p(q)^{1/q}}\|Z_n\|_q-1\bigg).
\end{align}
It follows from Lemma \ref{lem:ProbRep} that
\begin{align*}
V_n \overset{d}{=} {S_n^{(1)}\over qM_p(q)} - {S_n^{(2)}\over p} - {W_n\over p\,\sqrt{n}} + \sqrt{n}\,\Psi_p\Big({S_n^{(1)}\over\sqrt{n}},{S_n^{(2)}\over\sqrt{n}},{W_n\over n}\Big).
\end{align*}
For any $n\in\N$, we decompose $V_n$ into the random variables
\begin{align*}
T_n := {S_n^{(1)}\over qM_p(q)} - {S_n^{(2)}\over p} - {W_n\over p\,\sqrt{n}}\quad\text{and}\quad R_n:=\sqrt{n}\,\Psi_p\Big({S_n^{(1)}\over\sqrt{n}},{S_n^{(2)}\over\sqrt{n}},{W_n\over n}\Big)\,.
\end{align*}
Slutsky's theorem (see \cite[Proposition A.42 (b)]{BassStochasticProcesses}) completes the proof of Theorem \ref{thm:CLT} once we show that
\begin{align*}
T_n \todistr N\quad\text{and}\quad R_n\toprob 0,
\end{align*}
where $N\sim\mathcal{N}(0,\sigma^2)$ is the centered Gaussian random variable as in Theorem \ref{thm:CLT}.

Assumption \eqref{eq:AssumptionsWn}  says that $W_n/\sqrt{n}$ converges in distribution to $0$, as $n\to\infty$. Therefore, the multivariate central limit theorem  applied to $(S_n^{(1)}, S_n^{(2)})$ and the continuous mapping theorem yield
\begin{align*}
T_n = {S_n^{(1)}\over qM_p(q)} - {S_n^{(2)}\over p} - {W_n\over p\,\sqrt{n}} \,\,\todistr\,\, {\xi\over qM_p(q)}-{\eta\over p}\,,
\end{align*}
where $(\xi,\eta)$ is a centered Gaussian random vector in $\R^2$ with covariance matrix $\Sigma$ given by
\begin{align*}
\Sigma = \begin{pmatrix}
M_p(2q)-M_p(q)^2 & M_p(p+q)-M_p(q) \\
M_p(p+q)-M_p(q) & M_p(2p)-1
\end{pmatrix}.
\end{align*}
As a consequence, ${\xi\over qM_p(q)}-{\eta\over p}$ is a centered Gaussian random variable $N$ with variance
\begin{align*}
\Var\bigg({\xi\over qM_p(q)}-{\eta\over p}\bigg) &= {M_p(2q)-M_p(q)^2\over q^2M_p(q)^2} + {M_p(2p)-1\over p^2} - 2\,{M_p(p+q)-M_p(q)\over pqM_p(q)}\\
&={\Gamma({1\over p})\Gamma({2q+1\over p})\over q^2\Gamma({q+1\over p})^2}-{1\over p}-{1\over q^2}=\sigma^2.
\end{align*}

Finally, we shall argue that $R_n\toprob0$. To this end, we write
\begin{align*}
R_n = \sqrt{n}\,\bigg({(S_n^{(1)})^2\over n}+{(S_n^{(2)})^2\over n}+{W_n^2\over n^2}\bigg)\,{\Psi_p\Big({S_n^{(1)}\over\sqrt{n}},{S_n^{(2)}\over\sqrt{n}},{W_n\over n}\Big)\over\Big\|\Big({S_n^{(1)}\over\sqrt{n}},{S_n^{(2)}\over\sqrt{n}},{W_n\over n}\Big)\Big\|_2^2}\,.
\end{align*}
Since there exist $\delta, M>0$ such that $|\Psi_p(x,y,z)|\leq M\|(x,y,z)\|_2^2$ whenever $\|(x,y,z)\|_2^2<\delta$, we obtain
\begin{align*}
\Pro\left[{\Psi_p\Big({S_n^{(1)}\over\sqrt{n}},{S_n^{(2)}\over\sqrt{n}},{W_n\over n}\Big)\over\Big\|\Big({S_n^{(1)}\over\sqrt{n}},{S_n^{(2)}\over\sqrt{n}},{W_n\over n}\Big)\Big\|_2^2}>M\right] &\leq \Pro\bigg[\Big\|\Big({S_n^{(1)}\over\sqrt{n}},{S_n^{(2)}\over\sqrt{n}},{W_n\over n}\Big)\Big\|_2^2 > \delta\bigg]\\
&\leq \Pro\bigg[{S_n^{(1)}\over\sqrt{n}}>{\sqrt{\delta\over 3}}\,\bigg]+\Pro\bigg[{S_n^{(2)}\over\sqrt{n}}>{\sqrt{\delta\over 3}}\,\bigg]+\Pro\bigg[{W_n\over n}>{\sqrt{\delta\over 3}}\,\bigg].
\end{align*}
The weak law of large numbers ensures that, as $n\to\infty$, the first two probabilities converge to zero, while our assumption \eqref{eq:AssumptionsWn} on the random variables $W_n$ ensures that the last probability tends to zero as well. Thus, for any $\varepsilon>0$, we have that
\begin{align*}
\Pro\big[R_n>\varepsilon\big] &\leq \Pro\bigg[\sqrt{n}\Big({(S_n^{(1)})^2\over n}+{(S_n^{(2)})^2\over n}+{W_n^2\over n^2}\Big)>{\varepsilon\over M}\bigg] + \Pro\left[{\Psi_p\Big({S_n^{(1)}\over\sqrt{n}},{S_n^{(2)}\over\sqrt{n}},{W_n\over n}\Big)\over\Big\|\Big({S_n^{(1)}\over\sqrt{n}},{S_n^{(2)}\over\sqrt{n}},{W_n\over n}\Big)\Big\|_2^2}>M\right]\\
&= \Pro\bigg[{(S_n^{(1)})^2\over\sqrt{n}}+{(S_n^{(2)})^2\over\sqrt{n}}+{W_n^2\over n^{3/2}}>{\varepsilon\over M}\bigg] + \Pro\left[{\Psi_p\Big({S_n^{(1)}\over\sqrt{n}},{S_n^{(2)}\over\sqrt{n}},{W_n\over n}\Big)\over\Big\|\Big({S_n^{(1)}\over\sqrt{n}},{S_n^{(2)}\over\sqrt{n}},{W_n\over n}\Big)\Big\|_2^2}>M\right].
\end{align*}
Again by the weak law of large numbers we have that $S_n^{(1)}/\sqrt{n}$ and $S_n^{(2)}/\sqrt{n}$ both converge to zero in probability, as $n\to\infty$. Moreover, the central limit theorem implies that $S_n^{(1)}$ and $S_n^{(2)}$ converge in distribution to non-degenerate Gaussian random variables. Hence, Slutsky's theorem implies that $(S_n^{(1)})^2/\sqrt{n}$ and $(S_n^{(2)})^2/\sqrt{n}$ both converge to zero in distribution. Since the random variables are defined on the same probability space and because the limit is (almost surely) constant, we even have that $(S_n^{(1)})^2/\sqrt{n}$ and $(S_n^{(2)})^2/\sqrt{n}$ converge to zero in probability. Finally, $W_n^2/n^{3/2}$ also converges to zero in probability by our assumption \eqref{eq:AssumptionsWn}. Thus, the first probability in the last expression also converges to zero, while the second summand has already been treated before. As a consequence, we conclude that indeed
\[
R_n\toprob 0\,,
\]
which completes the argument. \hfill $\Box$

\subsection{Proof of the generalized central limit theorem (Theorem \ref{thm:CLT general Version2})}\label{subsec:ProofGeneralCLT}

Since the proof of Theorem \ref{thm:CLT general Version2} is very similar to the one of Theorem \ref{thm:CLT}, we restrict ourselves to the details that need to be adapted.

First of all, we recall that
\begin{align*}
S_n^{(1)} = {1\over\sqrt{n}}\sum_{i=1}^n\big(|Y_i|^q-M_p(q)\big)\qquad\text{and}\qquad S_n^{(2)} = {1\over\sqrt{n}}\sum_{i=1}^n\big(|Y_i|^p-1\big)\,.
\end{align*}
Then, following with minimal changes the proof of Lemma \ref{lem:ProbRep}, we obtain
\begin{align*}
\|Z_n\|_q
&\stackrel{d}{=} {\big(\sqrt{n}S_n^{(1)}+nM_p(q)\big)^{1/q}\over\big(\sqrt{n}S_n^{(2)}+n+W_n\big)^{1/p}}
=
n^{{1\over q}-{1\over p}}M_p(q)^{1/q}
{
\big(1+{S_n^{(1)}\over\sqrt{n}M_p(q)}\big)^{1/q}\over\big(1+{S_n^{(2)}\over\sqrt{n}}+{{W_n^*}\over {\sqrt n}} + {{\mu_n}\over n}\big)^{1/p}}\\
&=
n^{{1\over q}-{1\over p}}M_p(q)^{1/q}
{
\big(1+{S_n^{(1)}\over\sqrt{n}M_p(q)}\big)^{1/q}\over
\big(1+\frac{\mu_n}{n}\big)^{1/p} \bigg(1+{S_n^{(2)}\over(1+\frac{\mu_n}{n})\sqrt{n}}+{{W_n^*}\over {(1+\frac{\mu_n}{n})\sqrt n}}\bigg)^{1/p}}.
\end{align*}
We define for each $n\in\N$ the random variable
$$
V_{n}
:=
\sqrt{n}\Bigg({n^{{1\over p}-{1\over q}} \frac{(1+\frac{\mu_n}{n})^{1/p}}{ M_p(q)^{1/q}}}\|Z_n\|_q-1\Bigg)
\stackrel{d}{=}
\sqrt n \left({\big(1+{S_n^{(1)}\over\sqrt{n}M_p(q)}\big)^{1/q}\over
\bigg(1+{S_n^{(2)}\over(1+\frac{\mu_n}{n})\sqrt{n}}+{{W_n^*}\over {(1+\frac{\mu_n}{n})\sqrt n}}\bigg)^{1/p}} -1\right).
$$
In the same way as in the proof of Lemma \ref{lem:ProbRep}, one shows that $V_{n}\stackrel{d}{=}T_{n}+R_{n}$ with
\begin{align*}
T_n := {S_n^{(1)}\over qM_p(q)} - {S_n^{(2)}\over p(1+\frac{\mu_n}{n})} - {W_n^*\over p(1+\frac{\mu_n}{n})}\qquad\text{and}\qquad R_n:=\sqrt{n}\Psi_p\bigg(\frac{S_n^{(1)}}{\sqrt{n}},\frac{S_n^{(2)}}{\sqrt{n}(1+\frac{\mu_n}{n})},\frac{W_n^*}{\sqrt{n}(1+\frac{\mu_n}{n})}\bigg)\,.
\end{align*}
Thus, using Slutsky's theorem, we conclude the result of Theorem \ref{thm:CLT general Version2} once we have shown that
\begin{align}\label{eq:CLTGeneralProof0}
T_{n} \todistr \tilde N\qquad\text{and}\qquad R_{n}\toprob 0\,,
\end{align}
where $\tilde N\sim\mathcal{N}(0,\tilde\sigma^2)$ is the Gaussian random variable from the statement of Theorem \ref{thm:CLT general Version2}.


We start with the assertion on the sequence $T_n$. First of all, we notice that by Assumption \eqref{eq:AssumptionsWnVersion2}, the multivariate central limit theorem  applied to $(S_n^{(1)}, S_n^{(2)})$,
and the continuous mapping theorem,
\[
T_n = {S_n^{(1)}\over qM_p(q)} - {S_n^{(2)}\over p(1+\frac{\mu_n}{n})} - {W_n^*\over p(1+\frac{\mu_n}{n})} \todistr {\xi\over qM_p(q)}-{\eta\over p(1+\mu)} - \frac{\zeta}{p(1+\mu)} =: \tilde N\,,
\]
where $\zeta\sim\mathcal N(0,\tau^2)$ is independent of the centered Gaussian random vector $(\xi,\eta)$ in $\R^2$ with covariance matrix $\Sigma$ given by
\begin{align*}
\Sigma = \begin{pmatrix}
M_p(2q)-M_p(q)^2 & M_p(p+q)-M_p(q) \\
M_p(p+q)-M_p(q) & M_p(2p)-1
\end{pmatrix}.
\end{align*}
The limiting variable $\tilde N$ is centered Gaussian. To compute its variance, observe that
\begin{align*}
\Var\bigg({\xi\over qM_p(q)}-{\eta\over p(1+\mu)}\bigg) & = {M_p(2q)-M_p(q)^2\over q^2M_p(q)^2} + {M_p(2p)-1\over p^2(1+\mu)^2} - 2\,{M_p(p+q)-M_p(q)\over pqM_p(q)(1+\mu)}. 
\end{align*}
Thus, the limiting variance is given by
\begin{align*}
\tilde \sigma^2
&= \Var \tilde N = {M_p(2q)-M_p(q)^2\over q^2M_p(q)^2} + {M_p(2p)-1\over p^2(1+\mu)^2} - 2{M_p(p+q)-M_p(q)\over pqM_p(q)(1+\mu)} + \frac{\tau^2}{p^2(1+\mu)^2}\\
&=
{\Gamma({1\over p})\Gamma({2q+1\over p})\over q^2\Gamma({q+1\over p})^2}-{1\over q^2} + \frac{1}{p(1+\mu)^2} - \frac{2}{p(1+\mu)} +\frac{\tau^2}{p^2(1+\mu)^2},
\end{align*}
where the second line follows by recalling~\eqref{eq:Mpq} and performing computations with gamma functions.

To show that $R_{n}\toprob0$, as $n\to\infty$, we can in principle follow the lines of the proof of Theorem \ref{thm:CLT}, but we have to replace the terms $S_n^{(2)}\over\sqrt{n}$ and $W_n\over n$ there by $S_n^{(2)}\over\sqrt{n}(1+{\mu_n\over n})$ and $W_n^*\over \sqrt{n}(1+{\mu_n\over n})$, respectively. In particular, in a first step this results in showing that both sequences converge in distribution to $0$, that is for every fixed $\delta>0$,
$$
\Pro\bigg[{S_n^{(2)}\over\sqrt{n}(1+\frac{\mu_n}{n})}\geq{\sqrt{\delta\over 3}}\,\bigg]\to 0\qquad\text{and}\qquad\Pro\bigg[{W_n^*\over \sqrt{n}(1+\frac{\mu_n}{n})}\geq {\sqrt{\delta\over 3}}\,\bigg]\to 0\,,
$$
as $n\to\infty$. Both claims easily follow from the Slutsky theorem after recalling that both $S_n^{(2)}$ and $W_n^*$ converge in distribution to normal random variables, and that $1 + \frac{\mu_n}{n} \to 1+\mu$.

Moreover, in a second step one needs to argue that for any fixed $\varepsilon,M>0$,
\begin{align}\label{eq:CLTGeneralProof1}
\Pro\bigg[{(S_n^{(1)})^2\over \sqrt{n}}+{(S_n^{(2)})^2\over \sqrt{n}(1+\frac{\mu_n}{n})^2}+{(W_n^*)^2\over \sqrt{n}(1+\frac{\mu_n}{n})^2}>{\varepsilon\over M}\bigg]\to 0\,,
\end{align}
as $n\to\infty$. Recall that all three sequences $S_n^{(1)}, S_n^{(2)}, W_n^*$ converge in distribution to normal random variables. For the former two sequences, this follows from the central limit theorem, whereas the claim for $W_n^*$ is a consequence of our assumption~\eqref{eq:AssumptionsWnVersion2}. Again by a Slutsky-type argument, the sequences $(S_n^{(1)})^2/\sqrt n$, $(S_n^{(2)})^2/(\sqrt{n}(1+\frac{\mu_n}{n})^2)$ and $(W_n^*)^2/(\sqrt{n}(1+\frac{\mu_n}{n})^2)$  converge to zero in probability, hence so does their sum. This establishes \eqref{eq:CLTGeneralProof1} and hence \eqref{eq:CLTGeneralProof0}, which completes the proof of Theorem \ref{thm:CLT general Version2}. \hfill $\Box$

\subsection{Proof of the moderate deviations principle (Theorem \ref{thm:MDP})}\label{subsec:ProofMDP}

Let $(b_n)_{n\in\N}$ be a sequence of positive real numbers such that $b_n=\omega(1)$ and $b_n=o(\sqrt{n})$. As in the proof of the central limit theorem, we consider the sequence of random variables
\begin{align*}
V_n = \sqrt{n}\bigg({n^{{1/p}-{1/q}}\over M_p(q)^{1/q}}\|Z_n\|_q-1\bigg)
\end{align*}
and observe that Lemma \ref{lem:ProbRep} implies
\begin{align*}
{V_n\over b_n} &= {\sqrt{n}\over b_n}\bigg({n^{{1/p}-{1/q}}\over M_p(q)^{1/q}}\|Z_n\|_q-1\bigg)\\
&\overset{d}{=}{S_n^{(1)}\over b_n\,qM_p(q)}-{S_n^{(2)}\over p\,b_n}-{W_n\over p\,b_n\sqrt{n}}+{\sqrt{n}\over b_n}\Psi_p\Big({S_n^{(1)}\over\sqrt{n}},{S_n^{(2)}\over\sqrt{n}},{W_n\over n}\Big),
\end{align*}
where $\Psi_p:\R^3\to\R$ is such that $|\Psi_p(x,y,z)|\leq M\|(x,y,z)\|_2^2$ whenever $\|(x,y,z)\|_2^2<\delta$ for some $M,\delta>0$.

\medspace

Our strategy to prove the moderate deviations principle of Theorem \ref{thm:MDP} is as follows:
\begin{itemize}
\item[1.] We prove a bivariate moderate deviations principle for the sequence of rescaled random vectors $b_n^{-1}(S_n^{(1)},S_n^{(2)})$ in $\R^2$.
\item[2.] We apply the contraction principle to deduce a moderate deviations principle for the linear combination $S_n^{(1)}/(b_n\,qM_p(q))-S_n^{(2)}/(p\,b_n)$.
\item[3.] We show that the sequence of random variables $V_n/b_n$ is exponentially equivalent to the sequence formed in step 2.
\end{itemize}

We start with the first step of the proof.

\begin{lemma}[Bivariate MDP]\label{lem:bivariateMDP}
Fix $0<p< \infty$ and $0<q<\infty$ with $q<p$. Let $(b_n)_{n\in\N}$ be a sequence of positive real numbers such that $b_n=\omega(1)$ and $b_n=o(\sqrt{n})$ and consider the random vectors
\begin{align}\label{eq:DefSn}
S_n := {1\over\sqrt{n}}\sum_{i=1}^n\big(|Y_i|^q-M_p(q),|Y_i|^p-1\big).
\end{align}
Then the sequence of random vectors $S_n/b_n$ satisfies an MDP on $\R^2$ with speed $b_n^2$ and good rate function
$$
\rate_1(x,y)=-{p^{1-2q/p}\Gamma({1\over p})^2\over 2c_{p,q}}\,x^2-\Bigg({\Gamma({1\over p})\Gamma({1+2q\over p})\over 2c_{p,q}}-{\Gamma({1+q\over p})^2\over 2c_{p,q}}\Bigg)y^2+{p^{-q/p}\Gamma({1\over p})\Gamma({1+q\over p})\over c_{p,q}}\,xy\,,
$$
where $c_{p,q}:=(p+q^2)\Gamma({1+q\over p})^2-p\Gamma({1\over p})\Gamma({1+2q\over p})$.
\end{lemma}

%

\begin{proof}
First, we observe that $S_n$ is a sum of centered i.\,i.\,d.\,random vectors in $\R^2$ with covariance matrix
\begin{equation}\label{eq:CovarianceMatrix}
{\bf C} = \begin{pmatrix}
c_{11} & c_{12} \\ c_{21} & c_{22}
\end{pmatrix}
\end{equation}
given by
\begin{align*}
c_{11} &= \Var(|Y_1|^q-M_p(q)) = M_p(2q) - M_p(q)^2,\\
c_{22} &= \Var(|Y_1|^q-1) = M_p(2p)-1,\\
c_{12} = c_{21} &= \Cov(|Y_1|^q-M_p(q),|Y_1|^q-1) = M_p(p+q)-M_p(q).
\end{align*}
The moment generating function of the random vector $(|Y_1|^q-M_p(q),|Y_1|^p-1)$ on $\R^2$ is given by
$$
M(\lambda,\mu) = \frac{1}{2p^{1/p}\Gamma\big(1+\frac{1}{p}\big)}\int_{\R}e^{\lambda(|x|^q-M_p(q))+\mu(|x|^p-1)-{|x|^p\over p}}\,\dint x\,.
$$
Since $q<p$, the function $M$ is finite on $\R\times(-\infty,1/p)$, a set which contains the origin $(0,0)\in\R^2$ in its interior. Therefore, Lemma \ref{lem:MDPVector} (with the choice $s_n=\sqrt{n}b_n$ there) implies that the sequence of random variables $S_n/b_n$ satisfies an MDP on $\R^2$ with speed $b_n^2$ and good rate function
$$
\rate_1(x,y) = {1\over 2}\big\langle(x,y)^T,{\bf C}^{-1}(x,y)^T\big\rangle={1\over 2(c_{11}c_{22}-c_{12}^2)}\big(c_{22}x^2+c_{11}y^2-2c_{12}xy\big).
$$
Inserting the values for $c_{11},c_{22}$ and $c_{12}=c_{21}$, and simplifying the resulting expression proves the claim.
\end{proof}

\begin{rmk}
\rm In the previous proof we used our assumption that $q<p$ in order to verify the finiteness of certain exponential moments. As already discussed in Remark \ref{rem:MDPExMom} above, there exist version of the MDP for sums of independent random vectors not requiring the finiteness of such exponential moments. However, also when applying such weaker versions from \cite{ArconesMDP}, for example, the assumption that $q<p$ is in fact needed.
\end{rmk}

We continue with the second step and use the contraction principle to obtain an MDP the linear combinations of $S_n^{(1)}$ and $S_n^{(2)}$.

\begin{lemma}[MDP for the core term]\label{lem:MDPcoreterms}
Let $(b_n)_{n\in\N}$ be s sequence of positive real numbers such that
$b_n=\omega(1)$ and $b_n=o(\sqrt{n})$. Then the sequence of random variables
$$
{S_n^{(1)}\over b_n\,qM_p(q)}-{S_n^{(2)}\over p\,b_n}
$$
satisfies an MDP on $\R$ with speed $b_n^2$ and good rate function $\rate_2(t)=t^2/(2\sigma^2)$, where $\sigma^2$ is the constant in Theorem \ref{thm:CLT}.
\end{lemma}
\begin{proof}
Consider the continuous function
\begin{equation}\label{eq:DefFunctionG}
G:\R^2\to\R,\quad (x,y)\mapsto {x\over qM_p(q)}-{y\over p}
\end{equation}
and observe that, for each $n\in\N$, the random variable ${S_n^{(1)}\over b_n\,qM_p(q)}-{S_n^{(2)}\over p\,b_n}$ has the same distribution as $G(S_n/b_n)$, where $S_n$ was defined in Lemma \ref{lem:bivariateMDP}. Thus, the contraction principle (see Lemma \ref{lem:contraction principle}) implies the desired MDP with speed $b_n^2$ and good rate function
$$
\rate_2(t) = \inf\big\{\rate_1(x,y):(x,y)\in\R^2\text{ such that }G(x,y)=t\big\},\quad t\in\R.
$$
This optimization problem leads us to the Lagrangian
\[
\mathcal{L}(x,y,\lambda) = \rate_1(x,y) + \lambda(G(x,y)-t)
\]
and the Lagrange multiplier equations
\begin{enumerate}
\item[(i)] $\frac{c_{22}}{c_{11}c_{22}-c_{12}^2}\,x - \frac{c_{12}}{c_{11}c_{22}-c_{12}^2}\,y + \frac{\lambda}{qM_p(q)} = 0$,
\item[(ii)] $\frac{c_{11}}{c_{11}c_{22}-c_{12}^2}\,y - \frac{c_{12}}{c_{11}c_{22}-c_{12}^2}\,x - \frac{\lambda}{p} = 0$,
\item[(iii)] $\frac{x}{qM_p(q)} - \frac{y}{p} - t = 0$,
\end{enumerate}
where $c_{11}$, $c_{22}$, and $c_{12}$ are the entries of the covariance matrix given by \eqref{eq:CovarianceMatrix}.
This yields the critical value
\[
\lambda = t\cdot\left[ \frac{\frac{c_{12}}{p c_{11}} - \frac{1}{qM_p(q)} }{qM_p(q)\Big(\frac{c_{22}}{c_{11}c_{22}-c_{12}^2} - \frac{c_{12}^2}{c_{11}(c_{11}c_{22}-c_{12}^2)}\Big)} -
\frac{\frac{1}{p} - \frac{c_{12}}{qM_p(q) c_{22}} }{p\Big(\frac{c_{11}}{c_{11}c_{22}-c_{12}^2} - \frac{c_{12}^2}{c_{22}(c_{11}c_{22}-c_{12}^2)}\Big)}\right]^{-1}
\]
and from a direct (but tedious) computation, we obtain the explicit quadratic form of the rate function. We refrain from providing the details of the computation.
\end{proof}

We will now proceed with the third step and prove the exponential equivalence. In what follows, we let the random vectors $S_n$ be as in \eqref{eq:DefSn}, the random variables $V_n$ as in \eqref{eq:DefV_n}, and the function $G$ be given by \eqref{eq:DefFunctionG}.

\begin{lemma}[Exponential equivalence - MDP]\label{lem:ExpEquivalence}
Let $(b_n)_{n\in\N}$ be s sequence of positive real numbers such that $b_n=\omega(1)$ and $b_n=o(\sqrt{n})$. Then the sequences of random variables $G(S_n/b_n)$ and $V_n/b_n$ are exponentially equivalent.
\end{lemma}
\begin{proof}
We start by recalling that, for each $n\in\N$,
\begin{align*}
V_n \overset{d}{=} {S_n^{(1)}\over qM_p(q)} - {S_n^{(2)}\over p} - {W_n\over p\,\sqrt{n}} + \sqrt{n}\,\Psi_p\Big({S_n^{(1)}\over\sqrt{n}},{S_n^{(2)}\over\sqrt{n}},{W_n\over n}\Big)
\end{align*}
and
\begin{align*}
G\Big(\frac{S_n}{b_n}\Big) \overset{d}{=} {S_n^{(1)}\over b_nqM_p(q)} - {S_n^{(2)}\over b_np} \,.
\end{align*}
Let us fix $\varepsilon>0$. We observe that
\begin{align*}
\Pro\bigg[\Big|{V_n\over b_n}-G\Big({S_n\over b_n}\Big)\Big|>\varepsilon\bigg] \leq \Pro\bigg[{W_n\over pb_n\sqrt{n}}>{\varepsilon\over 2}\bigg]+\Pro\bigg[{\sqrt{n}\over b_n}\Big|\Psi_p\Big({S_n^{(1)}\over\sqrt{n}},{S_n^{(2)}\over\sqrt{n}},{W_n\over n}\Big)\Big|>{\varepsilon\over 2}\Big|\bigg],
\end{align*}
where we used that $W_n$ is a non-negative random variable for each $n\in\N$. The function $\Psi_p$ is the same as in Lemma \ref{lem:ProbRep}. Assumption \eqref{eq:AssumptionMDP} (with $\delta = p\frac{\varepsilon}{2}$ there) implies
$$
\limsup_{n\to\infty}{1\over b_n^2}\log\Pro\Big[{W_n\over pb_n\sqrt{n}}>{\varepsilon\over 2}\Big] = \limsup_{n\to\infty}{1\over b_n^2}\log\bW_n\Big(\Big({\varepsilon\over 2}\,p\,b_n\,\sqrt{n},\infty\Big)\Big) =-\infty.
$$
To discuss the second term, we first write
\begin{align*}
\Pro\bigg[{\sqrt{n}\over b_n}\Big|\Psi_p\Big({S_n^{(1)}\over\sqrt{n}},{S_n^{(2)}\over\sqrt{n}},{W_n\over n}\Big)\Big|>{\varepsilon\over 2}\Big|\bigg] \leq \Pro\bigg[\Big\|\Big({S_n^{(1)}\over\sqrt{n}},{S_n^{(2)}\over\sqrt{n}},{W_n\over n}\Big)\Big\|_2^2>{\varepsilon\over 2M}\bigg]+\Pro\Bigg[{\sqrt{n}\over b_n}{\big|\Psi_p\big({S_n^{(1)}\over\sqrt{n}},{S_n^{(2)}\over\sqrt{n}},{W_n\over n}\big)\big|\over \big\|\big({S_n^{(1)}\over\sqrt{n}},{S_n^{(2)}\over\sqrt{n}},{W_n\over n}\big)\big\|_2^2}>M\Bigg],
\end{align*}
where $M\in(0,\infty)$ is the parameter from Lemma \ref{lem:ProbRep}.
For the first summand in the previous expression, we obtain the estimate
\begin{align*}
\Pro\bigg[\Big\|\Big({S_n^{(1)}\over\sqrt{n}},{S_n^{(2)}\over\sqrt{n}},{W_n\over n}\Big)\Big\|_2^2>{\varepsilon\over 2M}\bigg] &= \Pro\bigg[{(S_n^{(1)})^2\over n}+{(S_n^{(2)})^2\over n}+{W_n^2\over n^2}>{\varepsilon\over 2M}\bigg]\\
&\leq \Pro\bigg[{S_n^{(1)}\over\sqrt{n}}>{\sqrt{\varepsilon\over 6M}}\,\bigg]+\Pro\bigg[{S_n^{(2)}\over\sqrt{n}}>\sqrt{{\varepsilon\over 6M}}\,\bigg]+\Pro\bigg[{W_n\over n}>\sqrt{{\varepsilon\over 6M}}\,\bigg].
\end{align*}
The first two terms both decay like $e^{-cn}$ for a suitable $c\in(0,\infty)$ by Cram\'er's theorem (see Lemma \ref{lem:cramer}). For the last term, we use again condition \eqref{eq:AssumptionMDP} (with $\delta=\sqrt{\varepsilon/(6M)}$ there) and obtain
\begin{align*}
\limsup_{n\to\infty}{1\over b_n^2}\log\Pro\bigg[{W_n\over n}>\sqrt{{\varepsilon\over 6M}}\,\bigg] &= \limsup_{n\to\infty}{1\over b_n^2}\log\bW_n\bigg(\Big(n\sqrt{{\varepsilon \over 6M}},\infty\Big)\bigg)\\
&\leq \limsup_{n\to\infty}{1\over b_n^2}\log\bW_n\bigg(\Big(b_n\sqrt{n}\sqrt{{\varepsilon \over 6M}},\infty\Big)\bigg)=-\infty,
\end{align*}
where we also used that $b_n=o(\sqrt{n})$. As a consequence,
\begin{align*}
\limsup_{n\to\infty}{1\over b_n^2}\log\Pro\bigg[\Big\|\Big({S_n^{(1)}\over\sqrt{n}},{S_n^{(2)}\over\sqrt{n}},{W_n\over n}\Big)\Big\|_2^2>{\varepsilon\over 2M}\bigg] \leq -c\limsup_{n\to\infty}{n\over b_n^2}=-\infty,
\end{align*}
since $b_n=\omega(1)$ and $b_n=o(\sqrt{n})$. Recalling the definition and the properties of the function $\Psi_p$ from Lemma \ref{lem:ProbRep}, we obtain for sufficiently large $n$
\begin{align*}
\Pro\Bigg[{\sqrt{n}\over b_n}{\big|\Psi_p\big({S_n^{(1)}\over\sqrt{n}},{S_n^{(2)}\over\sqrt{n}},{W_n\over n}\big)\big|\over \big\|\big({S_n^{(1)}\over\sqrt{n}},{S_n^{(2)}\over\sqrt{n}},{W_n\over n}\big)\big\|_2^2}>M\Bigg] &\leq \Pro\bigg[{b_n\over\sqrt{n}}\Big\|\Big({S_n^{(1)}\over\sqrt{n}},{S_n^{(2)}\over\sqrt{n}},{W_n\over n}\Big)\Big\|_2^2\geq \delta\bigg]\\
&\leq \Pro\bigg[{S_n^{(1)}\over\sqrt{n}}>\sqrt{{\delta\over3}}\,\bigg]+\Pro\bigg[{S_n^{(2)}\over\sqrt{n}}>\sqrt{{\delta\over3}}\,\bigg]+\Pro\bigg[{W_n\over n}>\sqrt{{\delta\over3}}\,\bigg],
\end{align*}
where we also used that $b_n=o(\sqrt{n})$. Again by Cram\'er's theorem (see Lemma \ref{lem:cramer}), the first two terms decay like $e^{-cn}$ for suitable $c\in(0,\infty)$ and their sum is bounded by $2e^{-cn}$ for sufficiently large $n$. Using this together with assumption \eqref{eq:AssumptionMDP}, we obtain
$$
\limsup_{n\to\infty}{1\over b_n^2}\log \Pro\Bigg[{\sqrt{n}\over b_n}{\big|\Psi_p\big({S_n^{(1)}\over\sqrt{n}},{S_n^{(2)}\over\sqrt{n}},{W_n\over n}\big)\big|\over \big\|\big({S_n^{(1)}\over\sqrt{n}},{S_n^{(2)}\over\sqrt{n}},{W_n\over n}\big)\big\|_2^2}>M\Bigg] =-\infty,
$$
where we used that $b_n=o(\sqrt{n})$. Putting everything together and using \cite[Lemma 1.2.15]{DZ}, we get
$$
\limsup_{n\to\infty}{1\over b_n^2}\log\Pro\bigg[\Big|{V_n\over n}-G\Big({S_n\over b_n}\Big)\Big|>\varepsilon\bigg] = -\infty\,.
$$
Since $\varepsilon>0$ was arbitrary, this shows the exponential equivalence that was claimed in the lemma.
\end{proof}

\begin{proof}[Proof of Theorem \ref{thm:MDP}]
The MDP is now a direct consequence of Lemma \ref{prop:exponentially equivalent} together with the MDP for the core term (see Lemma \ref{lem:MDPcoreterms}) and the exponential equivalence (see Lemma \ref{lem:ExpEquivalence}).
\end{proof}

\subsection{Proof of the large deviations principles (Theorem \ref{thm:LDP})}\label{subsec:ProofLDP}

In this last section we present the proof of the large deviations principles in Theorem \ref{thm:LDP}. On the way, we shall use some results we have obtained in \cite{KPT2019}. In what follows, we assume that for each $n\in\N$, $Y^{(n)}=(Y_1,\dots,Y_n)$ is a vector of independent $p$-generalized Gaussian random variables, and we assume that $(Y^{(n)})_{n\in\N}$ and $(W_n)_{n\in\N}$ are independent.

\subsubsection{The case $q<p$}
We start by recalling that, for each $n\in\N$, we have the distributional equality
\[
n^{1/p-1/q}\|Z_n\|_q \stackrel{d}{=} n^{1/p-1/q} {\|Y^{(n)}\|_q\over(\|Y^{(n)}\|_p^p+W_n)^{1/p}} = {\big({1\over n}\sum_{i=1}^n|Y_i|^q\big)^{1/q}\over\big({1\over n}\sum_{i=1}^n|Y_i|^p+{W_n\over n}\big)^{1/p}},
\]
see the proof of Lemma \ref{lem:ProbRep}. In the proof of Theorem 1.2 in \cite{KPT2019}, we have already seen that the sequence of random vectors
\begin{equation}\label{eq:LDPfromKPT}
\Big({1\over n}\sum_{i=1}^n|Y_i|^q,{1\over n}\sum_{i=1}^n|Y_i|^p\Big)
\end{equation}
satisfies an LDP on $\R^2$ with speed $n$ and a good rate function $\rate_1(t_1,t_2)$. More precisely, thanks to Cram\'er's theorem (see Lemma \ref{lem:cramer}) $\rate_1$ can be identified as the Legendre-Fenchel transform $\Lambda^*$ of the function
\begin{equation}\label{eq:Lambda}
\Lambda(t_1,t_2) = \log\int_0^\infty e^{t_1 x^q+(t_2-1/p)x^p}{\dint x\over p^{1/p}\Gamma(1+1/p)}.
\end{equation}
Since $(Y^{(n)})_{n\in\N}$ and $(W_n)_{n\in\N}$ are assumed to be independent and since $(W_n/n)_{n\in\N}$ satisfies an LDP with speed $n$ and good rate function $\rate_{\bW}$, the sequence of random vectors
$$
\Big({1\over n}\sum_{i=1}^n|Y_i|^q,{1\over n}\sum_{i=1}^n|Y_i|^p,{W_n\over n}\Big)
$$
satisfies an LDP on $\R^3$ with good rate function $\rate_2$ given by
$$
\rate_2(t_1,t_2,t_3) = \rate_1(t_1,t_2)+\rate_{\bW}(t_3),\quad (t_1,t_2,t_3)\in\R^3,
$$
where we used Lemma \ref{JointRateFunction}. Next, we consider the mapping
$$
F:(0,\infty)\times (0,\infty) \times [0,\infty) \to \R,\quad (t_1,t_2,t_3)\mapsto{t_1^{1/q}\over(t_2+t_3)^{1/p}},
$$
which is continuous on its domain. Clearly,
$$
F\Big({1\over n}\sum_{i=1}^n|Y_i|^q,{1\over n}\sum_{i=1}^n|Y_i|^p,{W_n\over n}\Big) \overset{d}{=} n^{1/p-1/q}\|Z_n\|_q
$$
for each $n\in\N$. Therefore, we can apply the contraction principle (see Lemma \ref{lem:contraction principle}) to conclude that $(n^{1/p-1/q}\|Z_n\|_q)_{n\in\N}$ satisfies an LDP with speed $n$ and good rate function $\rate_{{\bf Z},1}=\rate_2\circ F^{-1}$. This completes the argument.\hfill $\Box$

\begin{rmk}
\rm The assumption that $q<p$ was used only in disguise above and is behind the LDP for the sequence of random vectors in \eqref{eq:LDPfromKPT}. Indeed and as indicated above, the proof of this LDP is based on Cram\'er's theorem, which in turn requires finiteness of some exponential moments, or equivalently, that the origin is an interior point of the domain of the function $\Lambda$ defined in \eqref{eq:Lambda}. However, from the definition of this function it is clear that this can only be the case if $q<p$.
\end{rmk}

\subsubsection{The case $q>p$}
As was shown in the proof of \cite[Theorem 1.3]{KPT2019}, the sequence of random variables
\[
U_n:=\Big(\frac{1}{n}\sum_{i=1}^n |Y_i|^q\Big)^{1/q}
\]
satisfies an LDP with speed $n^{p/q}$ and good rate function
\begin{equation}\label{eq:RateFunctionS}
\rate_{{\bf U}}(x) = \begin{cases}
\frac{1}{p}\big(x^q-M_p(q)\big)^{p/q} & : x\geq M_p(q)^{1/q} \\
+\infty & : \text{otherwise}.
\end{cases}
\end{equation}
We will now prove that the two sequences $(U_n)_{n\in\N}$ and $(n^{1/p-1/q}\|Z_n\|_q)_{n\in\N}$ are exponentially equivalent.

\begin{lemma}[Exponential equivalence - LDP]\label{lem:ExpEquivalenceLDP}
The sequences $(U_n)_{n\in\N}$ and $(n^{1/p-1/q}\|Z_n\|_q)_{n\in\N}$ are exponentially equivalent with rate $n^{p/q}$.
\end{lemma}
\begin{proof}
As we have seen in the proof of Lemma \ref{lem:ProbRep}, one has that
\[
n^{1/p-1/q}\|Z_n\|_q \stackrel{d}{=} n^{1/p-1/q} {\|Y^{(n)}\|_q\over(\|Y^{(n)}\|_p^p+W_n)^{1/p}}
\]
for each $n\in\N$.
Let $\eta\in(0,\infty)$. Then, for every $\varepsilon\in(0,1)$, we obtain
\begin{align*}
& \Pro\left[\Bigg| \Big(\frac{1}{n}\sum_{i=1}^n |Y_i|^q\Big)^{1/q} - \frac{\Big(\frac{1}{n}\sum_{i=1}^n |Y_i|^q\Big)^{1/q}}{\Big(\frac{1}{n}\sum_{i=1}^n |Y_i|^p+\frac{W_n}{n}\Big)^{1/p}}  \Bigg| > \eta \right] \cr
& = \Pro\left[ \Big(\frac{1}{n}\sum_{i=1}^n |Y_i|^q\Big)^{1/q}\, \Bigg| 1 - \frac{1}{\Big(\frac{1}{n}\sum_{i=1}^n |Y_i|^p+\frac{W_n}{n}\Big)^{1/p}}  \Bigg| > \eta \right] \cr
& \leq \Pro\bigg[ \Big(\frac{1}{n}\sum_{i=1}^n |Y_i|^q\Big)^{1/q} > \frac{\eta}{\varepsilon}\bigg] +  \Pro\left[ \Bigg| 1 - \Big(\frac{1}{n}\sum_{i=1}^n |Y_i|^p+\frac{W_n}{n}\Big)^{-1/p}  \Bigg| > \varepsilon \right] \cr
& \leq \Pro\bigg[ \Big(\frac{1}{n}\sum_{i=1}^n |Y_i|^q\Big)^{1/q} > \frac{\eta}{\varepsilon}\bigg] +  \Pro\left[ 1 - \Big(\frac{1}{n}\sum_{i=1}^n |Y_i|^p+\frac{W_n}{n}\Big)^{-1/p} > \varepsilon \right] \cr
& \qquad\qquad\qquad + \Pro\left[ 1 - \Big(\frac{1}{n}\sum_{i=1}^n |Y_i|^p+\frac{W_n}{n}\Big)^{-1/p} < -\varepsilon \right]\,.
\end{align*}
Let us consider the second term. Write
$$
(1-\varepsilon)^{-p} = \Big({1\over 2}+{1\over 2}(1-\varepsilon)^{-p}\Big)+\Big({1\over 2}(1-\varepsilon)^{-p}-{1\over 2}\Big) =: A_1(\varepsilon) + A_2(\varepsilon)
$$
and note that $A_1(\varepsilon)>1$ and $A_2(\varepsilon)>0$ for all $\varepsilon\in(0,1)$.  This leads to the estimate
\begin{align*}
\Pro\left[ 1 - \Big(\frac{1}{n}\sum_{i=1}^n |Y_i|^p+\frac{W_n}{n}\Big)^{-1/p} > \varepsilon \right] & = \Pro \bigg[ \frac{1}{n}\sum_{i=1}^n |Y_i|^p+\frac{W_n}{n} > (1-\varepsilon)^{-p} \bigg] \cr
& \leq \Pro \bigg[ \frac{1}{n}\sum_{i=1}^n |Y_i|^p > A_1(\varepsilon) \bigg] + \Pro \bigg[ \frac{W_n}{n} > A_2(\varepsilon) \bigg]\,.
\end{align*}
By Cram\'er's theorem, the first term in the previous line decays exponentially like $e^{-c_1n}$, since $A_1(\varepsilon)>1$ for all $\varepsilon\in(0,1)$. In fact, the rate function in the corresponding LDP does not vanish in $O\setminus\{1\}$, where $O\subset \R$ is an open neighborhood of $1$, which implies that the constant $c_1$ stays strictly positive when letting $\varepsilon\to 0$. In combination with our Assumption \eqref{eq:AssumptionLDP} (applied with $\delta=A_2(\varepsilon)$) this shows that
\begin{align*}
&\limsup_{n\to\infty}{1\over n^{p/q}}\log \Pro\left[ 1 - \Big(\frac{1}{n}\sum_{i=1}^n |Y_i|^p+\frac{W_n}{n}\Big)^{-1/p} > \varepsilon \right] \\
&\qquad\leq \limsup_{n\to\infty}\Big(-\frac{c_1 n}{n^{p/q}}\Big)  + \limsup_{n\to\infty}{1\over n^{p/q}}\log \Pro \bigg[ \frac{W_n}{n} > A_1(\varepsilon) \bigg] \\
&\qquad= -\lim_{n\to\infty}{c_1 n\over n^{p/q}} + \limsup_{n\to\infty}{1\over n^{p/q}}\log \Pro \bigg[ \frac{W_n}{n} > A_1(\varepsilon) \bigg]\\
&\qquad =-\infty\,,
\end{align*}
In the first inequality above, we have used the elementary fact (see, e.g., \cite[Lemma 1.2.15]{DZ}) that for families of non-negative real numbers $a_1(\delta), a_2(\delta)$, $\delta>0$ one has that
\begin{align*}
\limsup_{\delta\to 0}\delta \log\big(a_1(\delta) + a_2(\delta)\big) &  = \max\Big\{ \limsup_{\delta\to 0}\delta\log a_1(\delta), \limsup_{\delta\to 0}\delta\log a_2(\delta) \Big\} \\
& \leq \limsup_{\delta\to 0}\delta\log a_1(\delta) + \limsup_{\delta\to 0}\delta\log a_2(\delta)\,.
\end{align*}
The remaining term
\[
\Pro\left[ 1 - \Big(\frac{1}{n}\sum_{i=1}^n |Y_i|^p+\frac{W_n}{n}\Big)^{-1/p} < -\varepsilon \right]
\]
can be treated in the same way.

Putting everything together, we obtain from the LDP for the sequence $(U_n)_{n\in\N}$ that
\begin{align*}
& \limsup_{n\to\infty} \frac{1}{n^{p/q}} \log \Pro\left[\Bigg| \Big(\frac{1}{n}\sum_{i=1}^n |Y_i|^q\Big)^{1/q} - \frac{\Big(\frac{1}{n}\sum_{i=1}^n |Y_i|^q\Big)^{1/q}}{\Big(\frac{1}{n}\sum_{i=1}^n |Y_i|^p+\frac{W_n}{n}\Big)^{1/p}}  \Bigg| > \eta \right] \cr
& \leq \limsup_{n\to\infty} \frac{1}{n^{p/q}} \log \Pro\bigg[ \Big(\frac{1}{n}\sum_{i=1}^n |Y_i|^q \Big)^{1/q} > \frac{\eta}{\varepsilon}\bigg] \cr
& = \begin{cases}
-\frac{1}{p}\Big(\big(\frac{\eta}{\varepsilon}\big)^q-M_p(q)\Big)^{p/q} & : \frac{\eta}{\varepsilon} \geq M_p(q)^{1/q} \cr
-\infty & : \text{otherwise}.
\end{cases}
\end{align*}
When $\varepsilon\to 0 $, the expression above tends to $-\infty$. Hence, the two sequences $(U_n)_{n\in\N}$ and $(n^{1/p-1/q}\|Z_n\|_q)_{n\in\N}$ are indeed exponentially equivalent.
\end{proof}

\begin{proof}[Proof of Theorem \ref{thm:LDP}, part (2) ]
The proof of is now a direct consequence of Lemma \ref{lem:ExpEquivalenceLDP} combined with the fact that $(U_n)_{n\in\N}$ satisfies an LDP with speed $n^{p/q}$ and good rate function $\rate_{\bf U}$ given by \eqref{eq:RateFunctionS}.
\end{proof}

\subsection*{Acknowledgement}
We would also like to thank Nicola Turchi for exchanges about the topics of this paper.\\
ZK has been supported by the German Research Foundation under Germany's Excellence Strategy EXC 2044 – 390685587, Mathematics M\"unster: Dynamics - Geometry - Structure. JP has been supported by a \textit{Visiting International Professor Fellowship} from the Ruhr University Bochum and its Research School PLUS, by the Austrian Science Fund (FWF) Project F5508-N26, which is part of the Special Research Program ``Quasi-Monte Carlo Methods: Theory and Applications'', and by the FWF Project P32405 ``Asymptotic Geometric Analysis and Applications''. ZK and CT have been supported by the DFG Scientific Network \textit{Cumulants, Concentration and Superconcentration}.

\bibliographystyle{plain}
\bibliography{limit_thm}

\end{document}